\documentclass[reqno]{amsart}

\usepackage{amsmath, amssymb}
\usepackage[mathscr]{eucal}
\usepackage{upgreek}
\usepackage{hyperref}
\hypersetup{colorlinks=true, linkcolor=blue, urlcolor=blue, citecolor=[rgb]{0, 0.8, 0}}

\theoremstyle{plain}
\newtheorem{theorem}{Theorem}[section]
\newtheorem{lemma}[theorem]{Lemma}

\theoremstyle{remark}
\newtheorem{remark}[theorem]{Remark}

\numberwithin{equation}{section}

\newcommand{\du}{\mathrm{d}}

\DeclareMathOperator{\im}{Im}
\DeclareMathOperator{\ind}{ind}

\DeclareMathOperator{\Trace}{\mathbf{Trace}}

\title{Essentially isospectral transformations and their applications}

\author{Namig J. Guliyev}
\address{Institute of Mathematics and Mechanics, Azerbaijan National Academy of Sciences, 9 B.~Vahabzadeh str., AZ1141, Baku, Azerbaijan.}
\email{njguliyev@gmail.com}

\subjclass[2010]{34A25, 34A55, 34B07, 34B24, 34C10, 34L20, 34L40, 37K35, 47A75, 47E05}

\keywords{Darboux transformation, one-dimensional Schr\"{o}dinger equation, boundary conditions dependent on the eigenvalue parameter, asymptotics, oscillation, inverse problems, regularized trace}

\begin{document}
\maketitle
\begin{abstract}
We define and study the properties of Darboux-type transformations between Sturm--Liouville problems with boundary conditions containing rational Herglotz--Nevanlinna functions of the eigenvalue parameter (including the Dirichlet boundary conditions). Using these transformations, we obtain various direct and inverse spectral results for these problems in a unified manner, such as asymptotics of eigenvalues and norming constants, oscillation of eigenfunctions, regularized trace formulas, and inverse uniqueness and existence theorems.
\end{abstract}

\tableofcontents

\section{Introduction} \label{sec:introduction}

We consider the one-dimensional Schr\"{o}dinger equation (the Sturm--Liouville equation in Liouville normal form)
\begin{equation} \label{eq:SL}
  -y''(x) + q(x)y(x) = \lambda y(x)
\end{equation}
and the boundary conditions
\begin{equation} \label{eq:boundary}
  \frac{y'(0)}{y(0)} = -f(\lambda), \qquad \frac{y'(\pi)}{y(\pi)} = F(\lambda),
\end{equation}
where $q \in \mathscr{L}_1(0, \pi)$ is real-valued and
\begin{equation} \label{eq:f_F}
  f(\lambda) = h_0 \lambda + h + \sum_{k=1}^d \frac{\delta_k}{h_k - \lambda}, \qquad F(\lambda) = H_0 \lambda + H + \sum_{k=1}^D \frac{\Delta_k}{H_k - \lambda}
\end{equation}
are rational Herglotz--Nevanlinna functions with real coefficients, i.e., $h_0, H_0 \ge 0$, $\delta_k, \Delta_k > 0$, $h_1 < \ldots < h_d$, $H_1 < \ldots < H_D$. We also include the case when the first (respectively, the second) boundary condition is Dirichlet by writing $f = \infty$ (respectively, $F = \infty$).

It is straightforward to verify that if a function $v$ without zeros is a fixed solution of the equation~(\ref{eq:SL}) with $\lambda$ replaced by $\mu$, then for any solution $y$ of this equation with $\lambda \ne \mu$, the function $\widehat{y} := y' - yv'/v$ is a solution of the same equation with the potential $q$ replaced by the potential $\widehat{q} := q - 2(v'/v)'$. Also, the function $\widehat{v} := 1/v$ is a solution of the above equation with $\lambda$ and $q$ replaced by $\mu$ and $\widehat{q}$ respectively (cf.~\cite[Lemma 5.1]{PT1987}). Moreover, by applying the same procedure with $\widehat{v}$ instead of $v$ to the latter potential one arrives at the original potential. This technique allows one to write the above differential expressions as
\begin{equation*}
  -\frac{\du^2}{\du x^2} + q(x) - \mu = \left( \frac{\du}{\du x} + \frac{v'}{v} \right) \left( -\frac{\du}{\du x} + \frac{v'}{v} \right)
\end{equation*}
and
\begin{equation*}
  -\frac{\du^2}{\du x^2} + \widehat{q}(x) - \mu = \left( -\frac{\du}{\du x} + \frac{v'}{v} \right) \left( \frac{\du}{\du x} + \frac{v'}{v} \right)
\end{equation*}
respectively. These transformations, called the \emph{direct} and \emph{inverse Darboux transformations}, and the corresponding factorizations play an important role in mathematical physics. For instance, such factorizations are used in supersymmetric quantum mechanics as a way of obtaining supersymmetric partner potentials ($q$ and $\widehat{q}$ in the above notation)~\cite{CKS2001}. The Darboux transformations and their generalizations also provide a powerful method of generating new exactly solvable potentials from known ones in the theory of nonlinear evolution equations~\cite{MS1991}. For Sturm--Liouville problems with constant (i.e., independent of the eigenvalue parameter) boundary conditions, these transformations were used in~\cite{DT1984}, \cite{IMT1984}, \cite{IT1983}, \cite{KC2009}, \cite{PT1987} to give a complete characterization of the isospectral sets of potentials. An operator-theoretic description of the method can be found in~\cite{D1978}, \cite{S1978}. See also~\cite{BBW2010} for a detailed historical overview and an operator-theoretic treatment of the Darboux transformations in the case of boundary conditions containing the eigenvalue parameter.

Sturm--Liouville problems with boundary conditions dependent on the eigenvalue parameter arise naturally in a variety of physical problems, including heat conduction, diffusion, vibration and electric circuit problems (see \cite{F1977} and the references therein). Churchill \cite{C1942} seems to be the first who applied the Darboux transformations to Sturm--Liouville problems with boundary conditions dependent on the eigenvalue parameter. Binding, Browne and Watson \cite{BBW2002a}, \cite{BBW2002b} studied problems with a constant boundary condition at one endpoint and a boundary condition of the form~(\ref{eq:boundary}), (\ref{eq:f_F}) at the other endpoint. They observed that if $y$ satisfies a boundary condition of the latter form at an endpoint, then $\widehat{y}$ satisfies a condition of the same form~(\ref{eq:boundary}) with some other function of the form~(\ref{eq:f_F}). An immediate consequence of this is that if $y$ is an eigenfunction of~(\ref{eq:SL})-(\ref{eq:boundary}) corresponding to an eigenvalue not equal to the above $\mu$, then $\widehat{y}$ is an eigenfunction of another problem of the same form corresponding to the same eigenvalue. Also, in some cases the above $\widehat{v}$ becomes an eigenfunction of the latter problem. In order to prove that all of its eigenfunctions can be obtained in this way---in other words, that these boundary value problems are \emph{essentially isospectral} in the sense that their spectra coincide (with the possible exception of the smallest eigenvalue of one of these problems)---these authors studied the oscillation properties of eigenfunctions. On the contrary, here we use the first-order asymptotics of the eigenvalues, and later deduce the oscillation properties (among other things) from the essential isospectrality of our transformations.

Another distinctive feature of the present work is that to each boundary condition of the form~(\ref{eq:boundary}), i.e. to each function of the form~(\ref{eq:f_F}), we associate its \emph{index} (an integer) and a monic polynomial, and express various spectral characteristics of boundary value problems of the form~(\ref{eq:SL})-(\ref{eq:boundary}) in terms of these indices and the coefficients of these polynomials, so that we are able to formulate our results without considering separate cases as it is usually done in the literature.

The paper is organized as follows. In Section~\ref{sec:preliminaries} we introduce the necessary notation and prove some preliminary lemmas. Section~\ref{sec:transformations} is devoted to transformations between rational Herglotz--Nevanlinna functions and between boundary value problems having such functions in their boundary conditions. In Subsection~\ref{ss:nevanlinna} we define a transformation between rational Herglotz--Nevanlinna functions and study its properties. In Subsections~\ref{ss:isospectral} and~\ref{ss:inverseisospectral} we define direct and inverse transformations between boundary value problems of the form~(\ref{eq:SL})-(\ref{eq:boundary}), study properties of the spectral data under these transformations, and show that these two transformations are, in a sense, inverses of each other. We apply these transformations in Section~\ref{sec:applications} to the solution of various direct and inverse spectral problems. In Subsection~\ref{ss:asymptotics} we obtain asymptotic formulas for the eigenvalues and the norming constants (see Subsection~\ref{ss:spectraldata} for the definition) of the problem~(\ref{eq:SL})-(\ref{eq:boundary}). In Subsection~\ref{ss:oscillation} we extend the Sturm oscillation theorem to boundary conditions of the form~(\ref{eq:boundary}). In Subsection~\ref{ss:trace} we apply our direct transformation to the calculation of the so-called \emph{regularized traces}. In Subsection~\ref{ss:inverseproblems} we provide necessary and sufficient conditions for two sequences of real numbers to be the eigenvalues and the norming constants of a problem of the form~(\ref{eq:SL})-(\ref{eq:boundary}). Subsection~\ref{ss:symmetric} is devoted to symmetric boundary value problems. In Subsections~\ref{ss:partialinformation} and \ref{ss:interior} we provide partial generalizations of the Hochstadt--Lieberman theorem and a theorem of Mochizuki and Trooshin to the case of boundary value problems of the form~(\ref{eq:SL})-(\ref{eq:boundary}).

\section{Preliminaries} \label{sec:preliminaries}

\subsection{Notation} \label{ss:notation}

First we introduce some necessary notation. We assign to each function $f$ of the form~(\ref{eq:f_F}) two polynomials $f_\uparrow$ and $f_\downarrow$ by writing this function as
$$
  f(\lambda) = \frac{f_\uparrow(\lambda)}{f_\downarrow(\lambda)},
$$
where
$$
  f_\downarrow(\lambda) := h'_0 \prod_{k=1}^d (h_k - \lambda), \qquad h'_0 := \begin{cases} 1 / h_0, & h_0 > 0, \\ 1, & h_0 = 0. \end{cases}
$$
We define the \emph{index} of $f$ as
$$
  \ind f := \begin{cases} 2 d + 1, & h_0 > 0, \\ 2 d, & h_0 = 0. \end{cases}
$$
This number equals $\deg f_\uparrow + \deg f_\downarrow$ when at least one of $h_0$ and $h$ is nonzero. If $f = \infty$ then we just set
$$
  f_\uparrow(\lambda) := -1, \qquad f_\downarrow(\lambda) := 0, \qquad \ind f := -1.
$$
It can easily be verified that each nonconstant function $f$ of the form~(\ref{eq:f_F}) is strictly increasing on any interval not containing any of its poles, and $f(\lambda) \to \pm\infty$ (respectively, $f(\lambda) \to h$) as $\lambda \to \pm\infty$ if its index is odd (respectively, even). We denote the smallest pole of $f$ (if it has any) by
$$
  \mathring{\boldsymbol{\uppi}}(f) := \begin{cases} h_1, & \ind f \ge 2, \\ +\infty, & \ind f \le 1, \end{cases}
$$
and the total number of poles of this function not exceeding $\lambda$ by
$$
  \boldsymbol{\Pi}_f(\lambda) := \sum_{\substack{1 \le k \le d \\ h_k \le \lambda}} 1.
$$

For every nonnegative integer $n$ we denote by $\mathscr{R}_n$ the set of rational functions of the form (\ref{eq:f_F}) with $\ind f = n$; we also introduce $\mathscr{R}_{-1} := \{ \infty \}$, which corresponds to the Dirichlet boundary condition. Then $\mathscr{R}_0$ consists of all constant functions, $\mathscr{R}_1$ consists of all increasing affine functions and so on. We also denote
$$
  \mathscr{R} := \bigcup_{n=-1}^\infty \mathscr{R}_n.
$$

To every $f \in \mathscr{R}$ we assign a monic polynomial
$$
  \boldsymbol{\upomega}_f(\lambda) := (-1)^{\left\lfloor \frac{\ind f}{2} \right\rfloor} \lambda f_\downarrow \left( \lambda^2 \right) - (-1)^{\left\lceil \frac{\ind f}{2} \right\rceil} f_\uparrow \left( \lambda^2 \right),
$$
where $\lfloor \cdot \rfloor$ and $\lceil \cdot \rceil$ are the usual floor and ceiling functions. We denote by $\omega_1$ and $\omega_2$ respectively the second and third coefficients of this polynomial:
$$
  \boldsymbol{\upomega}_f(\lambda) = \lambda^{\ind f + 1} + \omega_1 \lambda^{\ind f} + \omega_2 \lambda^{\ind f - 1} + \ldots.
$$
It is easy to see that $\omega_2$ coincides with the second coefficient of $(-1)^{\left\lceil \frac{\ind f}{2} \right\rceil + 1} f_\uparrow(\lambda)$ if $\ind f$ is odd and coincides with the second coefficient of $(-1)^{\left\lfloor \frac{\ind f}{2} \right\rfloor} f_\downarrow(\lambda)$ otherwise. The numbers $\Omega_1$ and $\Omega_2$ are defined similarly for $F$.

Let $\varphi(x, \lambda)$ and $\psi(x, \lambda)$ be the solutions of (\ref{eq:SL}) satisfying the initial conditions
\begin{equation} \label{eq:phi_psi}
  \varphi(0, \lambda) = f_\downarrow(\lambda), \quad \varphi'(0, \lambda) = -f_\uparrow(\lambda), \quad \psi(\pi, \lambda) = F_\downarrow(\lambda), \quad \psi'(\pi, \lambda) = F_\uparrow(\lambda).
\end{equation}
Then standard arguments (e.g., \cite[Theorem 1.1.1]{FY2001}) show that the eigenvalues of the boundary value problem (\ref{eq:SL})-(\ref{eq:boundary}) coincide with the zeros of the \emph{characteristic function}
$$\chi(\lambda) := F_\uparrow(\lambda) \varphi(\pi, \lambda) - F_\downarrow(\lambda) \varphi'(\pi, \lambda) = f_\downarrow(\lambda) \psi'(0, \lambda) + f_\uparrow(\lambda) \psi(0, \lambda),$$
are real and simple, and for each eigenvalue $\lambda_n$ there exists a unique number $\beta_n \ne 0$ such that
\begin{equation} \label{eq:beta}
  \psi(x, \lambda_n) = \beta_n \varphi(x, \lambda_n).
\end{equation}

We denote by $\mathscr{AC}[0, \pi]$ the set of absolutely continuous functions on $[0, \pi]$, and by $\mathscr{W}_2^1[0, \pi]$ the Sobolev space $\{ f \in \mathscr{AC}[0, \pi] \colon f' \in \mathscr{L}_2(0, \pi) \}$. In analogy with the notation $o(1 / n^\alpha)$, we use the notation
\begin{equation*}
  x_n = y_n + \ell_2 \left( \frac{1}{n^\alpha} \right)
\end{equation*}
to mean $\sum_{n = 0}^\infty \left| n^\alpha (x_n - y_n) \right|^2 < \infty$. Finally, we denote by $\mathscr{P}(q, f, F)$ the boundary value problem (\ref{eq:SL})-(\ref{eq:boundary}), and by $\mathring{\boldsymbol{\uplambda}}(q, f, F)$ the smallest eigenvalue of this problem.

\subsection{Hilbert space} \label{ss:hilbert}

We now introduce a Hilbert space and construct a self-adjoint operator in it in such a way that the boundary value problem (\ref{eq:SL})-(\ref{eq:boundary}) will be equivalent to the eigenvalue problem for this operator. The exact form of the operator, however, depends on the functions $f$ and $F$. When $h_0 > 0$ and $H_0 > 0$ we consider the Hilbert space $\mathcal{H} = \mathscr{L}_2(0,\pi) \oplus \mathbb{C}^{d+D+2}$ with inner product given by

$$\langle Y, Z \rangle := \int_0^{\pi} y(x) \overline{z(x)} \,\du x + \sum_{k=1}^{d} \frac{y_k \overline{z_k}}{\delta_k} + \frac{y_{d+1} \overline{z_{d+1}}}{h_0} + \sum_{k=1}^{D} \frac{\eta_k \overline{\zeta_k}}{\Delta_k} + \frac{\eta_{D+1} \overline{\zeta_{D+1}}}{H_0}$$
for
$$Y = \begin{pmatrix} y(x) \\ y_1 \\ \vdots \\ y_{d+1} \\ \eta_1 \\ \vdots \\ \eta_{D+1} \end{pmatrix}, \qquad Z = \begin{pmatrix} z(x) \\ z_1 \\ \vdots \\ z_{d+1} \\ \zeta_1 \\ \vdots \\ \zeta_{D+1} \end{pmatrix} \in \mathcal{H}.$$
In this space we define the operator
$$A(Y) := \begin{pmatrix} -y''(x) + q(x)y(x) \\ \delta_1 y(0) + h_1 y_1 \\ \vdots \\ \delta_{d} y(0) + h_{d} y_{d} \\ y'(0) + h y(0) - \sum_{k=1}^{d} y_k \\ H_1 \eta_1 - \Delta_1 y(\pi) \\ \vdots \\ H_{D} \eta_{D} - \Delta_{D} y(\pi) \\ y'(\pi) - H y(\pi) - \sum_{k=1}^{D} \eta_k \end{pmatrix}$$
with
\begin{multline*}
  D(A) := \left\{ Y \in \mathcal{H} \bigm| y, y' \in \mathscr{AC}[0,\pi],\ -y'' + qy \in \mathscr{L}_2(0,\pi), \right. \\
  \left. \vphantom{\bigm|} y_{d+1} = -h_0 y(0),\ \eta_{D+1} = H_0 y(\pi) \right\}.
\end{multline*}

When at least one of the numbers $h_0$, $H_0$ is zero, the following modifications are needed. We set $\mathcal{H} = \mathscr{L}_2(0,\pi) \oplus \mathbb{C}^{d+D+1}$ in the case when only one of these numbers equals zero, and $\mathcal{H} = \mathscr{L}_2(0,\pi) \oplus \mathbb{C}^{d+D}$ otherwise. If $h_0 = 0$ (respectively, $H_0 = 0$) we omit the $(d+2)$-th components (respectively, the last components) in the above paragraph, and replace the condition $y_{d+1} = -h_0 y(0)$ (respectively, $\eta_{D+1} = H_0 y(\pi)$) by the condition $y'(0) + h y(0) - \sum_{k=1}^{d} y_k = 0$ (respectively, $y'(\pi) - H y(\pi) - \sum_{k=1}^{D} \eta_k = 0$) in the definition of the domain of $A$. If $\ind f \le 0$ (respectively, $\ind F \le 0$), i.e., the first (respectively, the second) boundary condition is independent of the eigenvalue parameter, then there are no $y_k$ (respectively, $\eta_k$) components at all, and the condition $y'(0) = -h y(0)$ or $y(0) = 0$ (respectively, the condition $y'(\pi) = H y(\pi)$ or $y(\pi) = 0$) is added in the definition of the domain of $A$.

As in the case when only one of the boundary conditions depends on the eigenvalue parameter (see, e.g., \cite{BBW2002b}, \cite{F1977}), one can prove that the operator $A$ is self-adjoint, its spectrum is discrete and coincides with the set of eigenvalues of (\ref{eq:SL})-(\ref{eq:boundary}), and its eigenvectors
$$\Phi_n := \begin{pmatrix} \varphi(x, \lambda_n) \\ \frac{\delta_1}{\lambda_n - h_1} \varphi(0, \lambda_n) \\ \vdots \\ \frac{\delta_{d}}{\lambda_n - h_{d}} \varphi(0, \lambda_n) \\ -h_0 \varphi(0, \lambda_n) \\ \frac{\Delta_1}{H_1 - \lambda_n} \varphi(\pi, \lambda_n) \\ \vdots \\ \frac{\Delta_{D}}{H_{D} - \lambda_n} \varphi(\pi, \lambda_n) \\ H_0 \varphi(\pi, \lambda_n) \end{pmatrix}$$
are orthogonal.

\subsection{Spectral data} \label{ss:spectraldata}

We define the \emph{norming constants} as
\begin{equation*}
  \gamma_n := \| \Phi_n \|^2 = \int_0^{\pi} \varphi^2(x, \lambda_n) \,\du x + f'(\lambda_n) f_\downarrow^2(\lambda_n) + \frac{1}{\beta_n^2} F'(\lambda_n) F_\downarrow^2(\lambda_n).
\end{equation*}
Since the number $\lambda_n$ coincides with one of the poles of the function $f$ (respectively, $F$) if and only if $f_\downarrow(\lambda_n) = 0$ (respectively, $F_\downarrow(\lambda_n) = 0$), the expression on the right-hand side is well-defined in this case too. The numbers $\{ \lambda_n, \gamma_n \}_{n \ge 0}$ are called the \emph{spectral data} of the problem $\mathscr{P}(q, f, F)$. We denote by $\mathring{\boldsymbol{\upgamma}}(q, f, F)$ the first norming constant of the problem $\mathscr{P}(q, f, F)$ (i.e., the norming constant corresponding to the smallest eigenvalue $\mathring{\boldsymbol{\uplambda}}(q, f, F)$ of this problem).

\begin{lemma} \label{lem:chi_beta_gamma}
The following equality holds:
\begin{equation} \label{eq:chi_beta_gamma}
  \chi'(\lambda_n) = \beta_n \gamma_n.
\end{equation}
\end{lemma}
\begin{proof}
Using (\ref{eq:phi_psi}) and (\ref{eq:beta}) in the equality
$$\left. (\lambda - \lambda_n) \int_0^\pi \psi(x, \lambda) \varphi(x, \lambda_n) \,\du x = \left( \psi(x, \lambda) \varphi'(x, \lambda_n) - \psi'(x, \lambda) \varphi(x, \lambda_n) \right) \right|_0^{\pi}$$
we obtain
\begin{equation*}
\begin{split}
  \frac{\chi(\lambda)}{\lambda - \lambda_n} &= \int_0^\pi \psi(x, \lambda) \varphi(x, \lambda_n) \,\du x + \frac{F_\uparrow(\lambda)F_\downarrow(\lambda_n) - F_\downarrow(\lambda)F_\uparrow(\lambda_n)}{\beta_n (\lambda - \lambda_n)} \\
  &+ \frac{f_\downarrow(\lambda) - f_\downarrow(\lambda_n)}{\lambda - \lambda_n} \psi'(0, \lambda) + \frac{f_\uparrow(\lambda) - f_\uparrow(\lambda_n)}{\lambda - \lambda_n} \psi(0, \lambda).
\end{split}
\end{equation*}
As $\lambda \to \lambda_n$, this equality leads to (\ref{eq:chi_beta_gamma}).
\end{proof}

\begin{lemma} \label{lem:asymptotics}
The following first-order asymptotics holds:
$$
  \sqrt{\lambda_n} = n - \frac{\ind f + \ind F}{2} + O\left( \frac{1}{n} \right).
$$
\end{lemma}
\begin{proof}
We write $\varphi(x, \lambda)$ as
$$
  \varphi(x, \lambda) = f_\downarrow(\lambda) C(x, \lambda) - f_\uparrow(\lambda) S(x, \lambda),
$$
where $C(x, \lambda)$ and $S(x, \lambda)$ are the solutions of (\ref{eq:SL}) satisfying the initial conditions $C(0, \lambda) = S'(0, \lambda) = 1$ and $S(0, \lambda) = C'(0, \lambda) = 0$. Using the well-known estimates for $C(x, \lambda)$ and $S(x, \lambda)$ we calculate
\begin{align*}
  \varphi(\pi, \lambda) &= \left( \sqrt{\lambda} \right)^{\ind f} \left( \cos \left( \sqrt{\lambda} + \frac{\ind f}{2} \right) \pi + O\left( \frac{e^{|\im \sqrt{\lambda}\pi|}}{\sqrt{\lambda}} \right) \right), \\
  \varphi'(\pi, \lambda) &= -\left( \sqrt{\lambda} \right)^{\ind f + 1} \left( \sin \left( \sqrt{\lambda} + \frac{\ind f}{2} \right) \pi + O\left( \frac{e^{|\im \sqrt{\lambda}\pi|}}{\sqrt{\lambda}} \right) \right).
\end{align*}
Thus
$$
  \chi(\lambda) = \left( \sqrt{\lambda} \right)^{\ind f + \ind F + 1} \left( \sin \left( \sqrt{\lambda} + \frac{\ind f + \ind F}{2} \right) \pi + O\left( \frac{e^{|\im \sqrt{\lambda}\pi|}}{\sqrt{\lambda}} \right) \right).
$$
Finally, a standard argument involving Rouch\'{e}'s theorem concludes the proof.
\end{proof}

With this method one can in principle get sharper asymptotic formulas for the spectral data. But we will later obtain them in a much shorter way (see Theorem~\ref{thm:asymptotics}).

\subsection{Smallest eigenvalues and nonexistence of zeros} \label{ss:nozeros}

Define a partial order on the set $\mathscr{R}$ as follows: $f \preccurlyeq g$ if and only if either $f = \infty$, or $f$ and $g$ are two functions satisfying $f(\lambda) \le g(\lambda)$ for all $\lambda < \min \{ \mathring{\boldsymbol{\uppi}}(f), \mathring{\boldsymbol{\uppi}}(g) \}$.

\begin{lemma} \label{lem:lambda0}
If $f \preccurlyeq \widetilde{f}$ and $F \preccurlyeq \widetilde{F}$ then $\mathring{\boldsymbol{\uplambda}}(q, f, F) \ge \mathring{\boldsymbol{\uplambda}}(q, \widetilde{f}, \widetilde{F})$.
\end{lemma}
\begin{proof}
We only prove $\mathring{\boldsymbol{\uplambda}}(q, f, F) \ge \mathring{\boldsymbol{\uplambda}}(q, f, \widetilde{F})$; the proof of $\mathring{\boldsymbol{\uplambda}}(q, f, \widetilde{F}) \ge \mathring{\boldsymbol{\uplambda}}(q, \widetilde{f}, \widetilde{F})$ is similar. Denote $\nu_0 := \mathring{\boldsymbol{\uplambda}}(q, f, \infty)$. Dividing both sides of the identity
\begin{equation*}
\begin{split}
  \varphi(\pi, \lambda) \varphi'(\pi, \mu) &- \varphi'(\pi, \lambda) \varphi(\pi, \mu) \\
  &= f_\uparrow(\lambda) f_\downarrow(\mu) - f_\downarrow(\lambda) f_\uparrow(\mu) + (\lambda - \mu) \int_0^\pi \varphi(t, \lambda) \varphi(t, \mu) \,\du t
\end{split}
\end{equation*}
by $\mu - \lambda$ and taking the limit as $\mu \to \lambda$ we obtain
$$\frac{\du}{\du \lambda} \left( \frac{\varphi'(\pi, \lambda)}{\varphi(\pi, \lambda)} \right) = - \frac{1}{\varphi^2(\pi, \lambda)} \left( f_\downarrow^2(\lambda) \frac{\du f(\lambda)}{\du \lambda} + \int_0^\pi \varphi^2(t, \lambda) \,\du t \right) < 0$$
for $\lambda \in (-\infty, \nu_0)$. The proof of Lemma~\ref{lem:asymptotics} implies
$$\lim_{\lambda \to -\infty} \frac{\varphi'(\pi, \lambda)}{\varphi(\pi, \lambda)} = +\infty, \qquad \lim_{\lambda \to \nu_0-0} \frac{\varphi'(\pi, \lambda)}{\varphi(\pi, \lambda)} = -\infty.$$
Thus $\varphi'(\pi, \lambda)/\varphi(\pi, \lambda)$ is strictly monotone decreasing from $+\infty$ to $-\infty$ as $\lambda$ increases from $-\infty$ to $\nu_0$. This together with the fact that $\mathring{\boldsymbol{\uplambda}}(q, f, F)$ and $\mathring{\boldsymbol{\uplambda}}(q, f, \widetilde{F})$ are the smallest values of $\lambda$ for which $\varphi'(\pi, \lambda)/\varphi(\pi, \lambda) = F(\lambda)$ and $\varphi'(\pi, \lambda)/\varphi(\pi, \lambda) = \widetilde{F}(\lambda)$ respectively, concludes the proof.
\end{proof}

\begin{remark} \label{rem:pi_f}
The above proof also shows that $\mathring{\boldsymbol{\uplambda}}(q, f, F) < \min \{ \mathring{\boldsymbol{\uppi}}(f), \mathring{\boldsymbol{\uppi}}(F) \}$.
\end{remark}

\begin{lemma} \label{lem:no_zero}
If $\lambda \le \mathring{\boldsymbol{\uplambda}}(q, f, \infty)$ (respectively, $\lambda \le \mathring{\boldsymbol{\uplambda}}(q, \infty, F)$) then the function $\varphi(x, \lambda)$ (respectively, $\psi(x, \lambda)$) has no zeros in $(0,\pi)$.
\end{lemma}
\begin{proof}
Let $\nu_0$ be defined as in the proof of Lemma~\ref{lem:lambda0}. Since the function $\varphi(x, \nu_0)$ is an eigenfunction of the problem $\mathscr{P}(q, f, \infty)$, it is a constant multiple of the function $S_{\pi}(x, \nu_0)$, where $S_{\pi}(x, \lambda)$ is defined as the solution of (\ref{eq:SL}) satisfying the initial conditions $S_{\pi}(\pi, \lambda) = 0$ and $S'_{\pi}(\pi, \lambda) = 1$. It is well known that $S_{\pi}(x, \lambda)$ has no zeros in $(0,\pi)$ for $\lambda \le \mathring{\boldsymbol{\uplambda}}(q, \infty, \infty)$ (i.e., for values of $\lambda$ not greater than the smallest eigenvalue of the Dirichlet problem for (\ref{eq:SL})). But $\nu_0 \le \mathring{\boldsymbol{\uplambda}}(q, \infty, \infty)$ by Lemma~\ref{lem:lambda0}. Thus $S_{\pi}(x, \nu_0)$ and hence $\varphi(x, \nu_0)$ has no zeros in $(0,\pi)$.

Now suppose to the contrary that $\varphi(x, \lambda)$ has zeros in $(0,\pi)$ for some $\lambda \le \nu_0$. Let $x_0$ be its smallest positive zero. Remark~\ref{rem:pi_f} shows that $\varphi(0, \lambda) = f_\downarrow(\lambda) > 0$ and $\varphi(0, \nu_0) = f_\downarrow(\nu_0) > 0$. Thus $\varphi(x, \lambda) > 0$ and $\varphi(x, \nu_0) > 0$ for $x \in (0, x_0)$. Then $\varphi'(x_0, \lambda) < 0$, and hence
\begin{equation*}
\begin{split}
  0 &> \varphi(x_0, \nu_0) \varphi'(x_0, \lambda) - \varphi'(x_0, \nu_0) \varphi(x_0, \lambda) \\
  &= f_\downarrow(\lambda) f_\downarrow(\nu_0) \left( f(\nu_0) - f(\lambda) \right) + (\nu_0 - \lambda) \int_0^{x_0} \varphi(t, \nu_0) \varphi(t, \lambda) \,\du t > 0.
\end{split}
\end{equation*}
This contradiction proves the lemma for $\varphi$. The proof for $\psi$ is similar.
\end{proof}

\section{Transformations} \label{sec:transformations}

\subsection{Transformation of Herglotz--Nevanlinna functions} \label{ss:nevanlinna}

If we apply the Darboux transformation to eigenfunctions of the problem~(\ref{eq:SL})-(\ref{eq:boundary}), we obtain eigenfunctions of another problem of the same form having some other functions from $\mathscr{R}$ in its boundary conditions. We thus have a transformation between elements of $\mathscr{R}$. In this subsection we study such transformations.

We denote
$$
  \mathcal{S} := \left\{ (\mu, \tau, f) \in \mathbb{R} \times \mathbb{R} \times \mathscr{R} \colon \mu < \mathring{\boldsymbol{\uppi}}(f),\ \tau \ge f(\mu) \text{ if } \ind f \ge 0 \right\},
$$
and define the transformation
$$
  \boldsymbol{\Theta} \colon \mathcal{S} \to \mathscr{R},\ (\mu, \tau, f) \mapsto \widehat{f}
$$
by
$$
  \widehat{f}(\lambda) := \frac{\mu - \lambda}{f(\lambda) - \tau} - \tau.
$$
In the particular case when $f(\lambda) \equiv \tau$ (respectively, $f = \infty$) this is understood as $\widehat{f} := \infty$ (respectively, $\widehat{f}(\lambda) := -\tau$). One sees immediately from this definition that
\begin{equation} \label{eq:ThetaTheta}
  \boldsymbol{\Theta}(\mu, -\tau, \boldsymbol{\Theta}(\mu, \tau, f)) = f.
\end{equation}
The other main properties of this transformation are summarized in the following lemma.
\begin{lemma} \label{lem:f_hat}
The transformation $\boldsymbol{\Theta}$ is well-defined, i.e., $\widehat{f} := \boldsymbol{\Theta}(\mu, \tau, f) \in \mathscr{R}$. The poles of $f$ and $\widehat{f}$ interlace if $\ind f \ge 2$ and $\ind \widehat{f} \ge 2$ (i.e., if both $f$ and $\widehat{f}$ have poles); moreover, $\mathring{\boldsymbol{\uppi}}(f) < \mathring{\boldsymbol{\uppi}}(\widehat{f})$ if $\tau = f(\mu)$, and $\mathring{\boldsymbol{\uppi}}(f) > \mathring{\boldsymbol{\uppi}}(\widehat{f})$ if $\tau > f(\mu)$. Also, if $\tau = f(\mu)$ then $\ind \widehat{f} = \ind f - 1$,
\begin{equation} \label{eq:r_hat}
  \widehat{f}_\uparrow(\lambda) = \frac{-\tau f_\uparrow(\lambda) - \left( \lambda - \mu - \tau^2 \right) f_\downarrow(\lambda)}{\lambda - \mu}, \qquad \widehat{f}_\downarrow(\lambda) = \frac{f_\uparrow(\lambda) - \tau f_\downarrow(\lambda)}{\lambda - \mu},
\end{equation}
while if $\tau > f(\mu)$ then $\ind \widehat{f} = \ind f + 1$,
\begin{equation} \label{eq:r_hat2}
  \widehat{f}_\uparrow(\lambda) = \tau f_\uparrow(\lambda) + \left( \lambda - \mu - \tau^2 \right) f_\downarrow(\lambda), \qquad \widehat{f}_\downarrow(\lambda) = -f_\uparrow(\lambda) + \tau f_\downarrow(\lambda).
\end{equation}
\end{lemma}
\begin{proof}
We assume that $\ind f \ge 2$, since the cases $\ind f = -1$, $0$, $1$ can be verified very easily. We have
$$
  \widehat{f}(\lambda) = \frac{f_\downarrow(\lambda) (\lambda - \mu)}{\tau f_\downarrow(\lambda) - f_\uparrow(\lambda)} - \tau,
$$
where the polynomials $f_\uparrow$ and $f_\downarrow$, and thus $f_\downarrow$ and $\tau f_\downarrow - f_\uparrow$ have no common roots. When $\tau = f(\mu)$ the polynomial $\tau f_\downarrow(\lambda) - f_\uparrow(\lambda)$ is divisible by $\lambda - \mu$, and hence $\widehat{f}$ is a rational function with the set of poles $\{ \lambda \ne \mu \mid f(\lambda) = \tau \}$. Denote by $\widehat{d}$ the cardinality of the latter set, and by $\widehat{h}_1$, $\widehat{h}_2$, $\ldots$, $\widehat{h}_{\widehat{d}}$ these poles. Recall that $f$ is strictly increasing on each of the intervals $(-\infty, h_1)$, $(h_1, h_2)$, $\dots$, $(h_{d-1}, h_d)$, $(h_d, +\infty)$. Hence $\widehat{h}_k \in (h_k, h_{k+1})$ for $k = 1$, $\ldots$, $d-1$. If $\ind f = 2d$ then $f(\lambda) \nearrow h < f(\mu) = \tau$ as $\lambda \to +\infty$, and thus $\widehat{d} = d - 1$. Since the degree of the polynomial $\left( \tau f_\downarrow(\lambda) - f_\uparrow(\lambda) \right) / (\lambda - \mu)$ also equals $d - 1$, the function $\widehat{f}$ can be written as
\begin{equation} \label{eq:f_hat}
  \widehat{f}(\lambda) = \widehat{h}_0 \lambda + \widehat{h} + \sum_{k=1}^{\widehat{d}} \frac{\widehat{\delta}_k}{\widehat{h}_k - \lambda}.
\end{equation}
Also, since $\widehat{f}(\lambda) \to +\infty$ as $\lambda \to +\infty$, we obtain $\widehat{h}_0 > 0$, and since $f(\lambda) \nearrow \tau$ as $\lambda \nearrow \widehat{h}_k$, we obtain $\widehat{\delta}_k > 0$. Therefore $\widehat{f} \in \mathscr{R}$ with $\ind \widehat{f} = 2 \widehat{d} + 1 = \ind f - 1$. Finally, the identities~(\ref{eq:r_hat}) are obtained by considering the leading coefficients of the polynomials $\tau f_\uparrow + \left( \lambda - \mu - \tau^2 \right) f_\downarrow$ and $\tau f_\downarrow - f_\uparrow$. If $\ind f = 2d + 1$ then $\widehat{f}$ has one more pole in $(h_d, +\infty)$, i.e. $\widehat{d} = d$. Also, since $f(\lambda) / \lambda \to h_0$ as $\lambda \to +\infty$, we obtain that $\lim_{\lambda \to +\infty} \widehat{f}(\lambda)$ is finite, i.e. $\widehat{h}_0 = 0$, and $\ind \widehat{f} = 2 \widehat{d} = \ind f - 1$.

The case $\tau > f(\mu)$ can be analyzed in a similar way by taking into account the fact that the set of poles of $\widehat{f}$ is now $\{ \lambda \in \mathbb{R} \mid f(\lambda) = \tau \}$.
\end{proof}

\begin{remark}
When $\ind f \ge 1$, there exists a number $\nu \in \left( \mu, \mathring{\boldsymbol{\uppi}}(f) \right)$ with $f(\nu) = \tau$. Thus one would be tempted to define the above transformation by
$$
  \widehat{f}(\lambda) := \frac{\mu - \lambda}{f(\lambda) - f(\nu)} - f(\nu)
$$
for $\mu \le \nu$ in the general case, as is done in~\cite{BBW2002b}, but obviously one cannot obtain an increasing affine $\widehat{f}$ from a constant $f$ in that way.
\end{remark}

\subsection{Direct transformation between problems} \label{ss:isospectral}

We now introduce the first of our two essentially isospectral transformations between boundary value problems of the form~(\ref{eq:SL})-(\ref{eq:boundary}), and study properties of the spectral data under this transformation. Our transformation reduces the index of each boundary coefficient by one (if it is not already Dirichlet). Hence, by applying this transformation $\max \{ \ind f, \ind F \}$ number of times to a boundary value problem of the form (\ref{eq:SL})-(\ref{eq:boundary}), we will eventually arrive at a problem with boundary conditions independent of the eigenvalue parameter. It is worth mentioning here that in the case when $\ind f = \ind F = 0$ our transformation coincides with the transformation $+$ in \cite{IMT1984}, and in the case when $\ind f = -1$ (respectively, $\ind f = 0$) it coincides with the transformation $S_D$ (respectively, $S_N$) in \cite{BBW2002b}.

The domain $\widehat{\mathcal{S}}$ of our transformation consists of all possible boundary value problems of the form (\ref{eq:SL})-(\ref{eq:boundary}), excluding the case when both boundary conditions are Dirichlet:
$$
  \widehat{\mathcal{S}} := \left\{ (q, f, F) \colon q \in \mathscr{L}_1(0, \pi),\ f, F \in \mathscr{R},\ \ind f + \ind F \ge -1 \right\}.
$$
We define the transformation
$$
  \widehat{\mathbf{T}} \colon \widehat{\mathcal{S}} \to \mathscr{L}_1(0, \pi) \times \mathscr{R} \times \mathscr{R},\ (q, f, F) \mapsto (\widehat{q}, \widehat{f}, \widehat{F})
$$
by
\begin{equation} \label{eq:q_f_F_hat}
  \widehat{q} := q - 2 \left( \frac{v'}{v} \right)', \qquad \widehat{f} := \boldsymbol{\Theta} \left( \Lambda, -\frac{v'(0)}{v(0)}, f \right), \qquad \widehat{F} := \boldsymbol{\Theta} \left( \Lambda, \frac{v'(\pi)}{v(\pi)}, F \right),
\end{equation}
where
\begin{equation} \label{eq:mu}
  \Lambda := \begin{cases} \lambda_0, & f, F \ne \infty, \\ \lambda_0 - 2, & \text{otherwise} \end{cases} \qquad \text{and} \qquad v(x) := \begin{cases} \varphi(x, \Lambda), & f \ne \infty, \\ \psi(x, \Lambda), & f = \infty. \end{cases}
\end{equation}
Then Remark~\ref{rem:pi_f}, Lemmas~\ref{lem:lambda0}, \ref{lem:no_zero}, \ref{lem:f_hat} and the identity
\begin{equation} \label{eq:v}
  \left( \frac{v'(x)}{v(x)} \right)' = q(x) - \Lambda - \left( \frac{v'(x)}{v(x)} \right)^2
\end{equation}
imply that the transformation $\widehat{\mathbf{T}}$ is well-defined. Lemma~\ref{lem:f_hat} also shows that if $\ind f \ge 0$ then $\ind \widehat{f} = \ind f - 1$, and if $\ind f = -1$ then $\ind \widehat{f} = 0$. The same is true for $F$ and $\widehat{F}$.

To state the next theorem, we introduce the following notation: let
\begin{equation} \label{eq:I}
  I := \ind f - \ind \widehat{f} = \begin{cases} 1, & \ind f \ge 0, \\ -1, & \ind f = -1 \end{cases}
\end{equation}
and
\begin{equation} \label{eq:J}
  J := \frac{\ind f + \ind F}{2} - \frac{\ind \widehat{f} + \ind \widehat{F}}{2} = \begin{cases} 1, & \ind f, \ind F \ge 0, \\ 0, & \text{otherwise.} \end{cases}
\end{equation}

\begin{theorem} \label{thm:transformation}
If $\{ \lambda_n, \gamma_n \}_{n \ge 0}$ is the spectral data of the problem $\mathscr{P}(q, f, F)$ and $(\widehat{q}, \widehat{f}, \widehat{F}) = \widehat{\mathbf{T}} (q, f, F)$ then the spectral data of the transformed problem $\mathscr{P}(\widehat{q}, \widehat{f}, \widehat{F})$ is
$$
  \left\{ \lambda_n, \frac{\gamma_n}{(\lambda_n - \Lambda)^I} \right\}_{n \ge J}.
$$
\end{theorem}
\begin{proof}
A routine calculation shows that for every $n \ge J$ (i.e., $\lambda_n \ne \Lambda$) the function
$$
  \varphi'(x, \lambda_n) - \frac{v'(x)}{v(x)} \varphi(x, \lambda_n)
$$
is an eigenfunction of $\mathscr{P}(\widehat{q}, \widehat{f}, \widehat{F})$ corresponding to the eigenvalue $\lambda_n$. Hence the numbers $\lambda_n$ for $n \ge J$ are eigenvalues of this boundary value problem, and Lemma~\ref{lem:asymptotics} shows that there are no other eigenvalues.

For the part concerning the norming constants, we consider the cases $\ind f \ge 0$ and $\ind f = -1$ separately. In the former case we set
\begin{equation} \label{eq:phi_hat}
  \widehat{\varphi}_n(x) := \frac{1}{\Lambda - \lambda_n} \left( \varphi'(x, \lambda_n) - \frac{v'(x)}{v(x)} \varphi(x, \lambda_n) \right).
\end{equation}
Then $\widehat{\varphi}_n$ satisfies the initial condition $\widehat{\varphi}_n(0) = \widehat{f}_\downarrow(\lambda_n)$ and the identity
\begin{equation*}
  \widehat{\varphi}_n^2(x) = \frac{\left( \varphi(x, \lambda_n) \widehat{\varphi}_n(x) \right)'}{\Lambda - \lambda_n} + \frac{\varphi^2(x, \lambda_n)}{\lambda_n - \Lambda}.
\end{equation*}
Now if $\ind f \ge 1$ then from Lemma \ref{lem:f_hat} we obtain
\begin{equation*}
  \widehat{f}'(\lambda_n) \widehat{f}_\downarrow^2(\lambda_n) = \frac{f_\downarrow(\lambda_n) \widehat{f}_\downarrow(\lambda_n)}{\Lambda - \lambda_n} + \frac{f'(\lambda_n) f_\downarrow^2(\lambda_n)}{\lambda_n - \Lambda}.
\end{equation*}
Similarly in the case $\ind F \ge 1$ we have
\begin{equation*}
  \widehat{F}'(\lambda_n) \widehat{\varphi}_n^2(\pi) = \frac{\varphi(\pi, \lambda_n) \widehat{\varphi}_n(\pi)}{\lambda_n - \Lambda} + \frac{F'(\lambda_n) \varphi^2(\pi, \lambda_n)}{\lambda_n - \Lambda}.
\end{equation*}
Using the last three identities we calculate
\begin{equation*}
\begin{split}
  \widehat{\gamma}_n :=& \int_0^{\pi} \widehat{\varphi}_n^2(x) \,\du x + \widehat{f}'(\lambda_n) \widehat{f}_\downarrow^2(\lambda_n) + \widehat{F}'(\lambda_n) \widehat{\varphi}_n^2(\pi) \\
  =& \frac{1}{\lambda_n - \Lambda} \left( \int_0^{\pi} \varphi^2(x, \lambda_n) \,\du x + f'(\lambda_n) f_\downarrow^2(\lambda_n) + F'(\lambda_n) \varphi^2(\pi, \lambda_n) \right) \\
  =& \frac{\gamma_n}{\lambda_n - \Lambda}.
\end{split}
\end{equation*}
Since in the case when $\ind f = 0$ (respectively, $\ind F \le 0$) the second summands (respectively, the last summands) are absent, by definition, from the expressions for the norming constants and $\widehat{\varphi}_n(0) = 0$ (respectively, $\varphi(\pi, \lambda_n) \widehat{\varphi}_n(\pi) = 0$), the above relation between $\widehat{\gamma}_n$ and $\gamma_n$ holds also if $\ind f = 0$ or $\ind F \le 0$.

In the case $\ind f = -1$ we define $\widehat{\varphi}_n$ by
\begin{equation*}
  \widehat{\varphi}_n(x) := \varphi'(x, \lambda_n) - \frac{v'(x)}{v(x)} \varphi(x, \lambda_n).
\end{equation*}
Then $\widehat{\varphi}_n$ satisfies the initial condition $\widehat{\varphi}_n(0) = 1 \equiv \widehat{f}_\downarrow(\lambda_n)$,
and arguing as above, we establish the equality $\widehat{\gamma}_n = \gamma_n (\lambda_n - \Lambda)$.
\end{proof}

\begin{remark}
The motivation for choosing the values given in~(\ref{eq:mu}) for $\Lambda$ is due to the following observation. By choosing $v$ as an eigenfunction corresponding to the smallest eigenvalue we reduce the indices of both boundary coefficients. This is possible because of Lemmas~\ref{lem:lambda0} and \ref{lem:no_zero}. But if one of the boundary conditions is Dirichlet then this eigenfunction equals zero at an endpoint of the interval $[0, \pi]$. The above lemmas show that one can choose $\Lambda$ as any number strictly less than $\lambda_0$. The reason we do not choose $\lambda_0 - 1$ is that, as Theorem~\ref{thm:transformation} shows, the norming constant $\gamma_0$ is either multiplied or divided by $\lambda_0 - \Lambda$, depending on which one of the boundary conditions is Dirichlet. By choosing $\lambda_0 - 2$ (for definiteness) we will be able to determine this boundary condition in the next subsection.
\end{remark}

\subsubsection{An expression for $\mathring{\boldsymbol{\upgamma}}(q, f, F)$} \label{sss:gamma0}

We are now going to obtain an expression for the first norming constant $\gamma_0$ of the problem $\mathscr{P}(q, f, F)$ with $f, F \ne \infty$ in terms of the transformed problem $\mathscr{P}(\widehat{q}, \widehat{f}, \widehat{F})$, where $(\widehat{q}, \widehat{f}, \widehat{F}) := \widehat{\mathbf{T}} (q, f, F)$. This expression will be used in the next subsection to invert the action of $\widehat{\mathbf{T}}$.

Let $\widehat{C}(x, \lambda)$ and $\widehat{S}(x, \lambda)$ be the solutions of the equation
\begin{equation} \label{eq:SL_hat}
  -y''(x) + \widehat{q}(x)y(x) = \lambda y(x)
\end{equation}
satisfying the initial conditions
\begin{equation} \label{eq:C_S}
  \widehat{C}(0, \lambda) = \widehat{S}'(0, \lambda) = 1, \qquad \widehat{S}(0, \lambda) = \widehat{C}'(0, \lambda) = 0.
\end{equation}
It is easy to see that the function $1 / \varphi(x, \lambda_0)$ satisfies the equation (\ref{eq:SL_hat}) and the initial conditions
\begin{equation*}
  \frac{1}{\varphi(0, \lambda_0)} = \frac{1}{f_\downarrow(\lambda_0)}, \qquad \left( \frac{1}{\varphi(x, \lambda_0)} \right)'_{x=0} = \frac{f_\uparrow(\lambda_0)}{f_\downarrow^2(\lambda_0)} = \frac{f(\lambda_0)}{f_\downarrow(\lambda_0)}.
\end{equation*}
Thus
\begin{equation*}
  \frac{1}{\varphi(x, \lambda_0)} = \frac{1}{f_\downarrow(\lambda_0)} \left( \widehat{C}(x, \lambda_0) + f(\lambda_0) \widehat{S}(x, \lambda_0) \right).
\end{equation*}
Since $\widehat{S}(x, \lambda_0)$ and $1 / \varphi(x, \lambda_0)$ are both solutions of the equation (\ref{eq:SL_hat}), their Wronskian is constant:
\begin{equation*}
  \frac{\widehat{S}'(x, \lambda_0)}{\varphi(x, \lambda_0)} + \widehat{S}(x, \lambda_0) \frac{\varphi'(x, \lambda_0)}{\varphi^2(x, \lambda_0)} = \frac{\widehat{S}'(0, \lambda_0)}{\varphi(0, \lambda_0)} + \widehat{S}(0, \lambda_0) \frac{\varphi'(0, \lambda_0)}{\varphi^2(0, \lambda_0)} = \frac{1}{f_\downarrow(\lambda_0)},
\end{equation*}
and hence
\begin{equation*}
  \varphi^2(x, \lambda_0) = f_\downarrow(\lambda_0) \left( \widehat{S}(x, \lambda_0) \varphi(x, \lambda_0) \right)'.
\end{equation*}
If $\ind f \ge 1$ and $\ind F \ge 1$ then we have
\begin{equation*}
  f'(\lambda_0) = -\frac{1}{\widehat{f}(\lambda_0) + f(\lambda_0)}
\end{equation*}
and
\begin{equation*}
  F'(\lambda_0) = -\frac{1}{\widehat{F}(\lambda_0) + F(\lambda_0)} = - \left( \widehat{F}(\lambda_0) + \frac{\varphi'(\pi, \lambda_0)}{\varphi(\pi, \lambda_0)} \right)^{-1}.
\end{equation*}
Using the above identities and (\ref{eq:r_hat}) we calculate
\begin{equation*}
\begin{split}
  \gamma_0 &= \int_0^{\pi} \varphi^2(x, \lambda_0) \,\du x + f'(\lambda_0) f_\downarrow^2(\lambda_0) + F'(\lambda_0) \varphi^2(\pi, \lambda_0) \\
  &= f_\downarrow(\lambda_0) \widehat{S}(\pi, \lambda_0) \varphi(\pi, \lambda_0) - \frac{f_\downarrow^2(\lambda_0)}{\widehat{f}(\lambda_0) + f(\lambda_0)} - \varphi^2(\pi, \lambda_0) \left( \widehat{F}(\lambda_0) + \frac{\varphi'(\pi, \lambda_0)}{\varphi(\pi, \lambda_0)} \right)^{-1} \\
  &= \varphi(\pi, \lambda_0) \left( f_\downarrow(\lambda_0) \widehat{S}(\pi, \lambda_0) - \frac{\varphi^2(\pi, \lambda_0)}{\widehat{F}(\lambda_0) \varphi(\pi, \lambda_0) + \varphi'(\pi, \lambda_0)} \right) - \frac{f_\downarrow^2(\lambda_0)}{\widehat{f}(\lambda_0) + f(\lambda_0)} \\
  &= f_\downarrow(\lambda_0) \varphi^2(\pi, \lambda_0) \frac{\widehat{S}(\pi, \lambda_0) \widehat{F}(\lambda_0) - \widehat{S}'(\pi, \lambda_0)}{\widehat{F}(\lambda_0) \varphi(\pi, \lambda_0) + \varphi'(\pi, \lambda_0)} - \frac{f_\downarrow^2(\lambda_0)}{\widehat{f}(\lambda_0) + f(\lambda_0)} \\
  &= \frac{f_\downarrow^2(\lambda_0)}{\vphantom{\widehat{f}}\varkappa + f(\lambda_0)} - \frac{f_\downarrow^2(\lambda_0)}{\widehat{f}(\lambda_0) + f(\lambda_0)} = \left( \widehat{f}_\uparrow(\lambda_0) - \varkappa \widehat{f}_\downarrow(\lambda_0) \right) \frac{\widehat{f}_\uparrow(\lambda_0) + f(\lambda_0) \widehat{f}_\downarrow(\lambda_0)}{\varkappa + f(\lambda_0)},
\end{split}
\end{equation*}
where
\begin{equation*}
  \varkappa := \frac{\widehat{C}'(\pi, \lambda_0) - \widehat{C}(\pi, \lambda_0) \widehat{F}(\lambda_0)}{\widehat{S}'(\pi, \lambda_0) - \widehat{S}(\pi, \lambda_0) \widehat{F}(\lambda_0)}.
\end{equation*}
One can easily verify that this equality holds for the case $\ind f = 0$ too. If $\ind F = 0$ then $\widehat{F} = \infty$, and we only need to replace the above expression for $\varkappa$ by $\varkappa := \widehat{C}(\pi, \lambda_0) / \widehat{S}(\pi, \lambda_0)$.

\subsection{Inverse transformation between problems} \label{ss:inverseisospectral}

By applying the transformation $\widehat{\mathbf{T}}$ to a problem $\mathscr{P}(q, f, F)$ of the form (\ref{eq:SL})-(\ref{eq:boundary}) we obtain a new problem $\mathscr{P}(\widehat{q}, \widehat{f}, \widehat{F})$ of the same form. Now we want to restore the original problem $\mathscr{P}(q, f, F)$ from the transformed problem $\mathscr{P}(\widehat{q}, \widehat{f}, \widehat{F})$. As we will see below, in order to be able to determine the original problem we need some more information, e.g., the smallest eigenvalue $\lambda_0$ and the corresponding norming constant $\gamma_0$ of the problem $\mathscr{P}(q, f, F)$. But first we need to determine whether one of $f$ and $F$ is $\infty$ or not. Theorem~\ref{thm:transformation} shows that $\mathring{\boldsymbol{\uplambda}}(q, f, F) = \mathring{\boldsymbol{\uplambda}}(\widehat{q}, \widehat{f}, \widehat{F})$ if and only if one of the boundary conditions of the problem $\mathscr{P}(q, f, F)$ is Dirichlet. In this case the same theorem together with (\ref{eq:mu}) also tells us which of the two boundary conditions is Dirichlet, and the value of $v' / v$ at one of the endpoints of the interval $[0, \pi]$ can be immediately found from~(\ref{eq:q_f_F_hat}). In the case when none of the boundary conditions is Dirichlet, this value can be found from the expression for $\mathring{\boldsymbol{\upgamma}}(q, f, F)$ in terms of $\mathscr{P}(\widehat{q}, \widehat{f}, \widehat{F})$. Knowing this value, $v$ can be uniquely (up to a constant multiple) determined by the fact that $1 / v$ satisfies the equation (\ref{eq:SL_hat}) with $\lambda = \Lambda$.

With these considerations in mind, we define the transformation
$$\widetilde{\mathbf{T}} \colon \widetilde{\mathcal{S}} \to \mathscr{L}_1(0, \pi) \times \mathscr{R} \times \mathscr{R},\ (\mu, \nu, q, f, F) \mapsto (\widetilde{q}, \widetilde{f}, \widetilde{F})$$
on the union
\begin{equation*}
  \widetilde{\mathcal{S}} := \widetilde{\mathcal{S}}_1 \cup \widetilde{\mathcal{S}}_2 \cup \widetilde{\mathcal{S}}_3
\end{equation*}
of the sets
\begin{equation*}
  \widetilde{\mathcal{S}}_1 := \left\{ (\mu, \nu, q, f, F) \colon q \in \mathscr{L}_1(0, \pi),\ f, F \in \mathscr{R},\ \mu < \mathring{\boldsymbol{\uplambda}}(q, f, F),\ \nu > 0 \right\},
\end{equation*}
\begin{multline*}
  \widetilde{\mathcal{S}}_2 := \Big\{ (\mu, \nu, q, f, F) \colon q \in \mathscr{L}_1(0, \pi),\ f \in \mathscr{R}_0,\ F \in \mathscr{R}, \\
  \mu = \mathring{\boldsymbol{\uplambda}}(q, f, F),\ \nu = \mathring{\boldsymbol{\upgamma}}(q, f, F) / 2 \Big\}
\end{multline*}
and
\begin{multline*}
  \widetilde{\mathcal{S}}_3 := \Big\{ (\mu, \nu, q, f, F) \colon q \in \mathscr{L}_1(0, \pi),\ f \in \mathscr{R},\ F \in \mathscr{R}_0, \\
  \mu = \mathring{\boldsymbol{\uplambda}}(q, f, F),\ \nu = 2 \mathring{\boldsymbol{\upgamma}}(q, f, F) \Big\}
\end{multline*}
as follows. Let $(\mu, \nu, q, f, F) \in \widetilde{\mathcal{S}}$.

If $\mu < \mathring{\boldsymbol{\uplambda}}(q, f, F)$ we denote $\Lambda := \mu$ and
\begin{equation*}
  \varkappa := \frac{C'(\pi, \mu) - C(\pi, \mu) F(\mu)}{S'(\pi, \mu) - S(\pi, \mu) F(\mu)}
\end{equation*}
(in the case when $F = \infty$ this is understood as $\varkappa := C(\pi, \mu) / S(\pi, \mu)$). Lemma~\ref{lem:lambda0} implies $\mu < \mathring{\boldsymbol{\uplambda}}(q, \infty, F)$, and hence $\varkappa$ is well-defined. Arguing as in the proof of Lemma~\ref{lem:lambda0} we see that $\mathscr{P}(q, \varkappa, F)$ has only one eigenvalue not exceeding $\mathring{\boldsymbol{\uplambda}}(q, \infty, F)$, and hence $\mu = \mathring{\boldsymbol{\uplambda}}(q, \varkappa, F)$. The same proof also shows that if $f \ne \infty$ then
$$
  f(\mu) < f \left( \mathring{\boldsymbol{\uplambda}}(q, f, F) \right) = -\frac{\psi' \left( 0, \mathring{\boldsymbol{\uplambda}}(q, f, F) \right)}{\psi \left( 0, \mathring{\boldsymbol{\uplambda}}(q, f, F) \right)} < -\frac{\psi'(0, \mu)}{\psi(0, \mu)} = \varkappa.
$$
The function
\begin{equation*}
  \gamma(t) := \left( f_\uparrow(\mu) - \varkappa f_\downarrow(\mu) \right) \frac{f_\uparrow(\mu) + t f_\downarrow(\mu)}{\varkappa + t}
\end{equation*}
is strictly monotone decreasing from $+\infty$ to $0$ as $t$ increases from $-\varkappa$ to $-f(\mu)$. Thus there is a unique $\tau \in \left( -\varkappa, -f(\mu) \right)$ such that $\gamma(\tau) = \nu$. We denote by $u$ the solution of (\ref{eq:SL}) with $\lambda = \mu$, $u(0) = 1$ and $u'(0) = \tau$. Lemma~\ref{lem:lambda0} implies that $\mu = \mathring{\boldsymbol{\uplambda}}(q, \varkappa, F) \le \mathring{\boldsymbol{\uplambda}}(q, \varkappa, \infty) < \mathring{\boldsymbol{\uplambda}}(q, -\tau, \infty)$, and hence Lemma~\ref{lem:no_zero} shows that $u$ has no zeros on $[0,\pi]$. Moreover, if $F \ne \infty$ then using the asymptotics of the solutions $S$ and $S'$ we obtain that the denominator of the above expression for $\varkappa$ is strictly positive, and thus $\tau > -\varkappa$ implies $F(\mu) < u'(\pi) / u(\pi)$.

If $\mu = \mathring{\boldsymbol{\uplambda}}(q, f, F)$, then either $\nu = \mathring{\boldsymbol{\upgamma}}(q, f, F) / 2$ or $\nu = 2 \mathring{\boldsymbol{\upgamma}}(q, f, F)$. In the former case $f$ is constant and we denote $u := \varphi(x, \mu - 2)$. In the latter case $F$ is constant and this time we denote $u := \psi(x, \mu - 2)$. In both cases $u$ has no zeros on $[0,\pi]$ by Lemma~\ref{lem:no_zero}, and we set $\Lambda := \mu - 2$.

Finally, we define
$$
  \widetilde{q} := q - 2 \left( \frac{u'}{u} \right)', \qquad \widetilde{f} := \boldsymbol{\Theta} \left( \Lambda, -\frac{u'(0)}{u(0)}, f \right), \qquad \widetilde{F} := \boldsymbol{\Theta} \left( \Lambda, \frac{u'(\pi)}{u(\pi)}, F \right).
$$

Now we prove that, in a sense, the two transformations that we defined in this and the previous subsections are inverses of each other.

\begin{theorem} \label{thm:inverse}
The transformations $\widehat{\mathbf{T}}$ and $\widetilde{\mathbf{T}}$ are inverses of each other in the sense that if $(q, f, F) \in \widehat{\mathcal{S}}$ and $(\widehat{q}, \widehat{f}, \widehat{F}) = \widehat{\mathbf{T}}(q, f, F)$ then
$$
  \widetilde{\mathbf{T}} \left( \mathring{\boldsymbol{\uplambda}}(q, f, F), \mathring{\boldsymbol{\upgamma}}(q, f, F), \widehat{q}, \widehat{f}, \widehat{F} \right) = (q, f, F),$$
and conversely if $(\mu, \nu, q, f, F) \in \widetilde{\mathcal{S}}$ then $\widehat{\mathbf{T}} \widetilde{\mathbf{T}}(\mu, \nu, q, f, F) = (q, f, F)$.
\end{theorem}
\begin{proof}
Denote $\lambda_0 := \mathring{\boldsymbol{\uplambda}}(q, f, F)$ and $(\widetilde{q}, \widetilde{f}, \widetilde{F}) := \widetilde{\mathbf{T}} \left( \lambda_0, \mathring{\boldsymbol{\upgamma}}(q, f, F), \widehat{q}, \widehat{f}, \widehat{F} \right)$. If $f \ne \infty$ and $F \ne \infty$ then $\left( \lambda_0, \mathring{\boldsymbol{\upgamma}}(q, f, F), \widehat{q}, \widehat{f}, \widehat{F} \right) \in \widetilde{\mathcal{S}}_1$. Comparing the definition of $\widetilde{\mathbf{T}}$ with the expression for $\mathring{\boldsymbol{\upgamma}}(q, f, F)$ derived in Subsubsection~\ref{sss:gamma0}, we conclude that $\tau = f(\lambda_0)$. Thus the functions $f_\downarrow(\lambda_0) / v(x)$ and $u(x)$ satisfy
the equation (\ref{eq:SL_hat}) with $\lambda = \lambda_0$ and the same initial conditions, and hence $u(x) = f_\downarrow(\lambda_0) / v(x)$. Then
$$
  \frac{u'(x)}{u(x)} = -\frac{v'(x)}{v(x)},
$$
and thus
$$
  \widetilde{q}(x) = \widehat{q}(x) - 2\left( \frac{u'(x)}{u(x)} \right)' = q(x) - 2\left( \frac{v'(x)}{v(x)} \right)' - 2\left( \frac{u'(x)}{u(x)} \right)' = q(x).
$$
Finally, the identity~(\ref{eq:ThetaTheta}) implies
$$
\widetilde{f} = \boldsymbol{\Theta} \left( \lambda_0, -\frac{u'(0)}{u(0)}, \boldsymbol{\Theta} \left( \lambda_0, -\frac{v'(0)}{v(0)}, f \right) \right) = f
$$
and similarly $\widetilde{F} = F$. The remaining cases and the converse statement can be analyzed in an analogous manner.
\end{proof}

We can also prove an analogue of Theorem~\ref{thm:transformation} for the transformation $\widetilde{\mathbf{T}}$.
\begin{theorem} \label{thm:inverse_transformation}
If $\{ \lambda_n, \gamma_n \}_{n \ge 0}$ is the spectral data of the problem $\mathscr{P}(q, f, F)$ and $(\widetilde{q}, \widetilde{f}, \widetilde{F}) = \widetilde{\mathbf{T}}(\mu, \nu, q, f, F)$ then the spectral data of the problem $\mathscr{P}(\widetilde{q}, \widetilde{f}, \widetilde{F})$ is
$$
  \left\{ \lambda_n, \gamma_n (\lambda_n - \Lambda)^I \right\}_{n \ge -J},
$$
where $\lambda_{-1} := \mu$, $\gamma_{-1} := \nu$, and $I$ and $J$ are defined as
\begin{equation*}
  I := \ind \widetilde{f} - \ind f, \qquad J := \frac{\ind \widetilde{f} + \ind \widetilde{F}}{2} - \frac{\ind f + \ind F}{2}.
\end{equation*}
\end{theorem}
\begin{proof}
If $\mu < \mathring{\boldsymbol{\uplambda}}(q, f, F)$ (i.e., $J = 1$) then one can easily verify that the function $1 / u$ is an eigenfunction of $\mathscr{P}(\widetilde{q}, \widetilde{f}, \widetilde{F})$ corresponding to the eigenvalue $\mu$. It follows from the definition of $\boldsymbol{\Theta}$ that $\widetilde{f}(\mu) = \tau$. Comparison of the definition of $\widetilde{\mathbf{T}}$ with the expression for $\mathring{\boldsymbol{\upgamma}}(q, f, F)$ derived in Subsubsection~\ref{sss:gamma0} gives $\mathring{\boldsymbol{\upgamma}}(\widetilde{q}, \widetilde{f}, \widetilde{F}) = \nu$. The rest of the proof follows readily from Theorems~\ref{thm:transformation} and \ref{thm:inverse}.
\end{proof}

\section{Applications} \label{sec:applications}

\subsection{Asymptotics of eigenvalues and norming constants} \label{ss:asymptotics}

As mentioned earlier, it is possible to obtain sharper asymptotic formulas for the spectral data of the problem $\mathscr{P}(q, f, F)$ by following the method of the proof of Lemma~\ref{lem:asymptotics}. However, this method requires a large amount of calculation, and has already been done in the case of constant boundary conditions (see, e.g., \cite[Theorem 1.1.3 and Remark 1.1.2]{FY2001}). Our next theorem shows that the transformation $\widehat{\mathbf{T}}$ allows us to extend them to the case of boundary conditions (\ref{eq:boundary}) with much less calculation and write them in a unified manner. But first we start with a preliminary lemma.

\begin{lemma} \label{lem:q_omega_Omega}
If $(\widehat{q}, \widehat{f}, \widehat{F}) = \widehat{\mathbf{T}} (q, f, F)$ then 
$$
  \frac{1}{2} \int_0^\pi q(x) \,\du x + \omega_1 + \Omega_1 = \frac{1}{2} \int_0^\pi \widehat{q}(x) \,\du x + \widehat{\omega}_1 + \widehat{\Omega}_1,
$$
where $\widehat{\omega}_1$ and $\widehat{\Omega}_1$ are the second coefficients of the polynomials $\boldsymbol{\upomega}_{\widehat{f}}$ and $\boldsymbol{\upomega}_{\widehat{F}}$ respectively.
\end{lemma}
\begin{proof}
We consider only the case $\ind f$, $\ind F \ge 0$. The other cases when $f = \infty$ or $F = \infty$ can be analyzed in a similar way. If $h_0 > 0$ then $\widehat{f}$ is of the form (\ref{eq:f_hat}) with $\widehat{h}_0 = 0$ and
$$
  \widehat{h} = - \frac{1}{h_0} - f(\lambda_0).
$$
If $\ind f > 0$ and $h_0 = 0$ then
$$
  \widehat{h}_0 = \frac{1}{f(\lambda_0) - h}.
$$
Finally, if $\ind f = 0$ then $\widehat{f} = \infty$. In all these cases $\widehat{\omega}_1 = \omega_1 + f(\lambda_0)$. Similarly $\widehat{\Omega}_1 = \Omega_1 + F(\lambda_0)$. Hence (\ref{eq:q_f_F_hat}) implies
\begin{equation*}
\begin{split}
  & \frac{1}{2} \int_0^\pi \widehat{q}(x) \,\du x + \widehat{\omega}_1 + \widehat{\Omega}_1 \\
  & = \frac{1}{2} \int_0^\pi q(x) \,\du x - \frac{\varphi'(\pi, \lambda_0)}{\varphi(\pi, \lambda_0)} + \frac{\varphi'(0, \lambda_0)}{\varphi(0, \lambda_0)} + \omega_1 + f(\lambda_0) + \Omega_1 + F(\lambda_0) \\
  & = \frac{1}{2} \int_0^\pi q(x) \,\du x + \omega_1 + \Omega_1.
\end{split}
\end{equation*}
\end{proof}

We are now in a position to prove

\begin{theorem} \label{thm:asymptotics}
The spectral data of the problem $\mathscr{P}(q, f, F)$ have the asymptotics
$$\begin{aligned}
  \sqrt{\lambda_n} &= n - \frac{\ind f + \ind F}{2} + \frac{1}{\pi n} \left( \frac{1}{2} \int_0^\pi q(x) \,\du x + \omega_1 + \Omega_1 \right) + \frac{\xi_n}{n}, \\
  \gamma_n &= \frac{\pi}{2} \left( n - \frac{\ind f + \ind F}{2} \right)^{2 \ind f} \left( 1 + \frac{\xi'_n}{n} \right),
\end{aligned}$$
where $\{\xi_n\}$, $\{\xi'_n\} = o(1)$ if $q \in \mathscr{L}_1(0, \pi)$ and $\{\xi_n\}$, $\{\xi'_n\} \in \ell_2$ if $q \in \mathscr{L}_2(0, \pi)$.
\end{theorem}
\begin{proof}
We give the proof for the case $q \in \mathscr{L}_1(0, \pi)$; the case when $q \in \mathscr{L}_2(0, \pi)$ differs from it only in the form of the remainder terms. Consider the chain of problems $\mathscr{P}(q^{(k)}, f^{(k)}, F^{(k)})$ defined by
\begin{equation} \label{eq:P_k}
\begin{aligned}
  (q^{(0)}, f^{(0)}, F^{(0)}) &:= (q, f, F), \\
  (q^{(k)}, f^{(k)}, F^{(k)}) &:= \widehat{\mathbf{T}} (q^{(k-1)}, f^{(k-1)}, F^{(k-1)}), \qquad k = 1, 2, \ldots, K,
\end{aligned}
\end{equation}
where $K := \max \{ \ind f, \ind F \}$, and let $\omega^{(k)}_1$ and $\Omega^{(k)}_1$ denote the second coefficients of the polynomials $\boldsymbol{\upomega}_{f^{(k)}}$ and $\boldsymbol{\upomega}_{F^{(k)}}$ respectively. Then the last problem $\mathscr{P}(q^{(K)}, f^{(K)}, F^{(K)})$ has constant boundary conditions, and hence its eigenvalues have the asymptotics
\begin{equation*}
\begin{split}
  \sqrt{\lambda_n^{(K)}} = n &- \frac{\ind f^{(K)} + \ind F^{(K)}}{2} \\
  &+ \frac{1}{\pi n} \left( \frac{1}{2} \int_0^\pi q^{(K)}(x) \,\du x + \omega^{(K)}_1 + \Omega^{(K)}_1 \right) + o \left( \frac{1}{n} \right).
\end{split}
\end{equation*}
Let $I$ and $J$ be defined by (\ref{eq:I})-(\ref{eq:J}) with $f$ and $F$ replaced by $f^{(K-1)}$ and $F^{(K-1)}$ respectively. Using Theorem~\ref{thm:transformation} and Lemma~\ref{lem:q_omega_Omega} we calculate
\begin{equation*}
\begin{split}
  \sqrt{\lambda_n^{(K-1)}} &= \sqrt{\lambda_{n-J}^{(K)}} \\
  &= n - J - \frac{\ind f^{(K)} + \ind F^{(K)}}{2} \\
  &\phantom{{}= n{}} {}+ \frac{1}{\pi (n - J)} \left( \frac{1}{2} \int_0^\pi q^{(K)}(x) \,\du x + \omega^{(K)}_1 + \Omega^{(K)}_1 \right) + o \left( \frac{1}{n} \right) \\
  &= n - \frac{\ind f^{(K-1)} + \ind F^{(K-1)}}{2} \\
  &\phantom{{}= n{}} {}+ \frac{1}{\pi n} \left( \frac{1}{2} \int_0^\pi q^{(K-1)}(x) \,\du x + \omega^{(K-1)}_1 + \Omega^{(K-1)}_1 \right) + o \left( \frac{1}{n} \right),
\end{split}
\end{equation*}
where we used the obvious relation
$$
  \frac{1}{\pi (n - J)} = \frac{1}{\pi n} + O \left( \frac{1}{n^2} \right).
$$
Repeating this argument $K-1$ more times yields the above asymptotics for $\sqrt{\lambda_n}$.

In a similar manner, from
$$
  \gamma_n^{(K)} = \frac{\pi}{2} \left( n - \frac{\ind f^{(K)} + \ind F^{(K)}}{2} \right)^{2 \ind f^{(K)}} \left( 1 + o \left( \frac{1}{n} \right) \right),
$$
Theorem~\ref{thm:transformation} and Lemma~\ref{lem:asymptotics} we obtain
\begin{equation*}
\begin{split}
  \gamma_n^{(K-1)} &= \gamma_{n-J}^{(K)} \left( \lambda_n^{(K-1)} - \Lambda \right)^I \\
  &= \frac{\pi}{2} \left( n - J - \frac{\ind f^{(K)} + \ind F^{(K)}}{2} \right)^{2 \ind f^{(K)}} \\
  &\phantom{{}= \frac{\pi}{2}{}} \times \left( n - \frac{\ind f^{(K-1)} + \ind F^{(K-1)}}{2} \right)^{2I} \left( 1 + o \left( \frac{1}{n} \right) \right) \\
  &= \frac{\pi}{2} \left( n - \frac{\ind f^{(K-1)} + \ind F^{(K-1)}}{2} \right)^{2 \ind f^{(K-1)}} \left( 1 + o \left( \frac{1}{n} \right) \right).
\end{split}
\end{equation*}
Again, repeating this argument $K-1$ more times we get the above asymptotics for the sequence $\gamma_n$.
\end{proof}

\subsection{Oscillation of eigenfunctions} \label{ss:oscillation}

The Sturm oscillation theorem says that an eigenfunction corresponding to the $n$-th eigenvalue of the Sturm--Liouville problem with constant boundary conditions has exactly $n$ zeros in the open interval $(0, \pi)$ (see, e.g., \cite[Theorem 1.2.2]{FY2001}). Oscillation properties of the eigenfunctions of problems with boundary conditions dependent on the eigenvalue parameter have been studied, e.g., in~\cite[Appendix I]{BP1981}, \cite[Section 3]{BBW2002a}. By using the transformation $\widehat{\mathbf{T}}$, we will now extend these results to boundary value problems of the form (\ref{eq:SL})-(\ref{eq:boundary}). But first we need the following auxiliary result.

\begin{lemma} \label{lem:oscillation}
Let $J$ and $\widehat{\varphi}_n$ be defined by the formulas (\ref{eq:J}) and (\ref{eq:phi_hat}) respectively. If the function $\widehat{\varphi}_n(x)$ has $N$ zeros in $(0, \pi)$ then the function $\varphi(x, \lambda_n)$ has exactly $N + J + \boldsymbol{\Pi}_{\widehat{f}}(\lambda_n) + \boldsymbol{\Pi}_{\widehat{F}}(\lambda_n) - \boldsymbol{\Pi}_f(\lambda_n) - \boldsymbol{\Pi}_F(\lambda_n)$ zeros in $(0, \pi)$.
\end{lemma}
\begin{proof}
We give the proof for the case $f \ne \infty$; in the case when $f = \infty$ we only need to consider $\psi$ instead of $\varphi$. Let $\Lambda$ be defined by~(\ref{eq:mu}). The identities
$$
  \left( \widehat{\varphi}_n(x) \varphi(x, \Lambda) \right)' = \varphi(x, \lambda_n) \varphi(x, \Lambda), \qquad \left( \frac{\varphi(x, \lambda_n)}{\varphi(x, \Lambda)} \right)' = (\Lambda - \lambda_n) \frac{\widehat{\varphi}_n(x)}{\varphi(x, \Lambda)}
$$
imply that between any two zeros of the function $\widehat{\varphi}_n(x)$ there is a zero of the function $\varphi(x, \lambda_n)$ and vice versa. If we denote by $x_1$, $\ldots$, $x_N$ the zeros of the function $\widehat{\varphi}_n(x)$ in $(0, \pi)$ then the function $\varphi(x, \lambda_n)$ has $N-1$ zeros in $(x_1, x_N)$. Using the equalities
$$\widehat{\varphi}_n(0) \varphi(0, \Lambda) = \widehat{f}_\downarrow(\lambda_n) f_\downarrow(\Lambda), \qquad \varphi(0, \lambda_n) \varphi(0, \Lambda) = f_\downarrow(\lambda_n) f_\downarrow(\Lambda)$$
and the above identities one can easily check that $\varphi(x, \lambda_n)$ has a zero in $(0, x_1)$ if and only if $\widehat{f}_\downarrow(\lambda_n) f_\downarrow(\lambda_n) > 0$ or $\widehat{f}_\downarrow(\lambda_n) = 0$, i.e., if and only if the functions $f$ and $\widehat{f}$ have the same number of poles not exceeding $\lambda_n$. A similar assertion holds for the interval $(x_N, \pi)$ and the functions $F$ and $\widehat{F}$ if the boundary condition at $\pi$ is not Dirichlet (i.e., $J = 1$). Otherwise, if the boundary condition at $\pi$ is Dirichlet (i.e., $J = 0$), the function $\varphi(x, \lambda_n)$ does not have a zero in $(x_N, \pi)$, but $\boldsymbol{\Pi}_F(\lambda_n) = \boldsymbol{\Pi}_{\widehat{F}}(\lambda_n) = 0$. This concludes the proof.
\end{proof}

Now we are ready to prove our main oscillation result.

\begin{theorem} \label{thm:oscillation}
An eigenfunction of the problem $\mathscr{P}(q, f, F)$ corresponding to the eigenvalue $\lambda_n$ has exactly $n - \boldsymbol{\Pi}_f(\lambda_n) - \boldsymbol{\Pi}_F(\lambda_n)$ zeros in $(0, \pi)$.
\end{theorem}
\begin{proof}
Consider the problems $\mathscr{P}(q^{(k)}, f^{(k)}, F^{(k)})$ defined by (\ref{eq:P_k}). Let $J^{(k)}$ be defined by (\ref{eq:J}) with $f$ and $F$ replaced by $f^{(k)}$ and $F^{(k)}$ respectively. Since the last problem $\mathscr{P}(q^{(K)}, f^{(K)}, F^{(K)})$ has constant boundary conditions, its eigenfunction corresponding to the eigenvalue $\lambda_m^{(K)}$ has $m$ zeros in the open interval $(0, \pi)$ for each $m \ge 0$. On the other hand, the constancy of $f^{(K)}$ and $F^{(K)}$ implies $\boldsymbol{\Pi}_{f^{(K)}}(\lambda) \equiv 0$ and $\boldsymbol{\Pi}_{F^{(K)}}(\lambda) \equiv 0$, and hence the statement of the theorem holds in this case. By successive applications of Theorem~\ref{thm:transformation}, it follows that $\lambda_n = \lambda_{n - J'}^{(K)}$, where $J' := \sum_{k=0}^{K-1} J^{(k)}$.
Applying Lemma~\ref{lem:oscillation} successively to the problems $\mathscr{P}(q^{(K-1)}, f^{(K-1)}, F^{(K-1)})$, $\ldots$, $\mathscr{P}(q^{(0)}, f^{(0)}, F^{(0)})$, we obtain that an eigenfunction of $\mathscr{P}(q, f, F)$ corresponding to the eigenvalue $\lambda_n$ has
\begin{multline*}
  n - J' + \sum_{k=0}^{K-1} \left( J^{(k)} + \boldsymbol{\Pi}_{f^{(k+1)}}(\lambda_n) + \boldsymbol{\Pi}_{F^{(k+1)}}(\lambda_n) - \boldsymbol{\Pi}_{f^{(k)}}(\lambda_n) - \boldsymbol{\Pi}_{F^{(k)}}(\lambda_n) \right) \\ = n - \boldsymbol{\Pi}_f(\lambda_n) - \boldsymbol{\Pi}_F(\lambda_n)
\end{multline*}
zeros in $(0, \pi)$.
\end{proof}

\subsection{Regularized trace formulas} \label{ss:trace}

In this subsection, we apply our direct transformation to the calculation of regularized traces. We refer to~\cite{SP2006} for a relatively recent survey on this topic. Regularized traces of Sturm--Liouville problems with boundary conditions dependent on the eigenvalue parameter have been calculated in~\cite{EE2012}, \cite{G2005b}, \cite{K2006}.

Throughout this subsection we assume that $q \in \mathscr{W}_2^1[0, \pi]$. As in the case of constant boundary conditions (see, e.g., \cite[Remark 1.1.1]{FY2001}, \cite[Appendix II]{LG1964} or \cite[Theorem 1.5.1]{M1977}), one can obtain more precise asymptotics for the spectral data of $\mathscr{P}(q, f, F)$, depending on the smoothness of the potential $q$. In particular, if $q \in \mathscr{W}_2^1[0, \pi]$ then the eigenvalues of the problem $\mathscr{P}(q, f, F)$ have the asymptotics
\begin{equation*}
  \sqrt{\lambda_n} = n - a + \frac{b}{n - a} + \ell_2 \left( \frac{1}{n^2} \right),
\end{equation*}
where
\begin{equation*}
  a := \frac{\ind f + \ind F}{2}, \qquad b := \frac{1}{\pi} \left( \frac{1}{2} \int_0^\pi q(x) \,\du x + \omega_1 + \Omega_1 \right).
\end{equation*}
Hence the following series (called \emph{the first regularized trace}) converges:
$$
  \Trace(q, f, F) := \sum_{n < a} \lambda_n + \sum_{n = a} (\lambda_n - b) + \sum_{n > a} \left( \lambda_n - (n - a)^2 - 2b \right).
$$
The sum of this series has already been calculated in~\cite{EE2012}. Here we express $s_\lambda$ in terms of $q$, $\omega_1$, $\omega_2$, $\Omega_1$ and $\Omega_2$, and give another proof of these formulas, based on the use of the transformation $\widehat{\mathbf{T}}$.

Again, we begin with a preliminary lemma.

\begin{lemma} \label{lem:omega_2}
Let $(\widehat{q}, \widehat{f}, \widehat{F}) := \widehat{\mathbf{T}} (q, f, F)$, and let $\widehat{\omega}_1$ and $\widehat{\omega}_2$ be the second and third coefficients of the polynomial $\boldsymbol{\upomega}_{\widehat{f}}$. We have
\begin{equation*}
  \frac{(-1)^{\ind \widehat{f}} \widehat{q}(0)}{4} - \frac{\widehat{\omega}_1^2}{2} - \widehat{\omega}_2 = \frac{(-1)^{\ind f} q(0)}{4} - \frac{\omega_1^2}{2} - \omega_2 \mp \frac{\Lambda}{2},
\end{equation*}
where $\Lambda$ is defined by~(\ref{eq:mu}), and the plus sign is used if and only if $f = \infty$. A similar identity holds for the right endpoint.
\end{lemma}
\begin{proof}
The proof of Lemma~\ref{lem:q_omega_Omega} shows that $\widehat{\omega}_1 = \omega_1 + f(\Lambda)$. If $f \ne \infty$ then (\ref{eq:q_f_F_hat}), (\ref{eq:mu}) and (\ref{eq:v}) imply
$$
  \widehat{q}(0) = -q(0) + 2 \Lambda + 2 f^2(\Lambda).
$$
We start with the case when $h_0 > 0$ (i.e., $\ind f$ is odd and positive). From the first identity of (\ref{eq:r_hat}) we get $\widehat{\omega}_2 = \omega_2 - \omega_1 f(\Lambda) + \Lambda$. Thus
\begin{equation*}
\begin{split}
  & \frac{(-1)^{\ind \widehat{f}} \widehat{q}(0)}{4} - \frac{\widehat{\omega}_1^2}{2} - \widehat{\omega}_2 \\
  & = - \frac{q(0)}{4} + \frac{\Lambda + f^2(\Lambda)}{2} - \frac{\omega_1^2}{2} - \omega_1 f(\Lambda) - \frac{f^2(\Lambda)}{2} - \omega_2 + \omega_1 f(\Lambda) - \Lambda \\
  & = - \frac{q(0)}{4} - \frac{\omega_1^2}{2} - \omega_2 - \frac{\Lambda}{2}.
\end{split}
\end{equation*}
In the case when $h_0 = 0$ (i.e., $\ind f$ is even) from the second identity of (\ref{eq:r_hat}) we get $\widehat{\omega}_2 = \omega_2 - \omega_1 f(\Lambda) - f^2(\Lambda)$. Hence
\begin{equation*}
\begin{split}
  & \frac{(-1)^{\ind \widehat{f}} \widehat{q}(0)}{4} - \frac{\widehat{\omega}_1^2}{2} - \widehat{\omega}_2 \\
  & = \frac{q(0)}{4} - \frac{\Lambda + f^2(\Lambda)}{2} - \frac{\omega_1^2}{2} - \omega_1 f(\Lambda) - \frac{f^2(\Lambda)}{2} - \omega_2 + \omega_1 f(\Lambda) + f^2(\Lambda) \\
  & = \frac{q(0)}{4} - \frac{\omega_1^2}{2} - \omega_2 - \frac{\Lambda}{2}.
\end{split}
\end{equation*}

Finally, if $f = \infty$ then $\widehat{f}$ is constant, and thus $\widehat{\omega}_1 = -\widehat{f}$ and $\widehat{\omega}_2 = \omega_1 = \omega_2 = 0$. From (\ref{eq:q_f_F_hat}), (\ref{eq:mu}) and (\ref{eq:v}) we have
$$
  \widehat{q}(0) = -q(0) + 2 \Lambda + 2 \widehat{f}^2.
$$
Therefore
\begin{equation*}
  \frac{(-1)^{\ind \widehat{f}} \widehat{q}(0)}{4} - \frac{\widehat{\omega}_1^2}{2} - \widehat{\omega}_2 = - \frac{q(0)}{4} + \frac{\Lambda + \widehat{f}^2}{2} - \frac{\widehat{f}^2}{2} = - \frac{q(0)}{4} + \frac{\Lambda}{2}.
\end{equation*}
\end{proof}

\begin{theorem} \label{thm:trace}
The following identity holds:
$$
  \Trace(q, f, F) = \frac{(-1)^{\ind f} q(0)}{4} + \frac{(-1)^{\ind F} q(\pi)}{4} - \frac{\omega_1^2}{2} - \frac{\Omega_1^2}{2} - \omega_2 - \Omega_2.
$$
\end{theorem}
\begin{proof}
Consider again the problems $\mathscr{P}(q^{(k)}, f^{(k)}, F^{(k)})$ defined by (\ref{eq:P_k}). Since the last problem $\mathscr{P}(q^{(K)}, f^{(K)}, F^{(K)})$ has constant boundary conditions, the identity in the statement of the theorem holds for this problem. We now consider the problem $\mathscr{P}(q^{(K-1)}, f^{(K-1)}, F^{(K-1)})$. If $\ind f^{(K-1)}$, $\ind F^{(K-1)} \ge 0$ then $\Lambda^{(K-1)} = \lambda_0^{(K-1)}$. In this case $\mathscr{P}(q^{(K-1)}, f^{(K-1)}, F^{(K-1)})$ has the extra eigenvalue $\lambda_0^{(K-1)}$ and hence
$$\Trace(q^{(K-1)}, f^{(K-1)}, F^{(K-1)}) = \lambda_0^{(K-1)} + \Trace(q^{(K)}, f^{(K)}, F^{(K)}).$$
Using Lemma~\ref{lem:omega_2} we calculate
\begin{equation*}
\begin{split}
  \Trace(q^{(K-1)}&, f^{(K-1)}, F^{(K-1)}) = \Trace(q^{(K)}, f^{(K)}, F^{(K)}) + \lambda_0^{(K-1)} \\
  &= \frac{(-1)^{\ind f^{(K)}} q^{(K)}(0)}{4} - \frac{\left( \omega_1^{(K)} \right)^2}{2} - \omega_2^{(K)} + \frac{\lambda_0^{(K-1)}}{2} \\
  &\phantom{\Trace} + \frac{(-1)^{\ind F^{(K)}} q^{(K)}(\pi)}{4} - \frac{\left( \Omega_1^{(K)} \right)^2}{2} - \Omega_2^{(K)} + \frac{\lambda_0^{(K-1)}}{2} \\
  &= \frac{(-1)^{\ind f^{(K-1)}} q^{(K-1)}(0)}{4} - \frac{\left( \omega_1^{(K-1)} \right)^2}{2} - \omega_2^{(K-1)} \\
  &\phantom{\Trace} + \frac{(-1)^{\ind F^{(K-1)}} q^{(K-1)}(\pi)}{4} - \frac{\left( \Omega_1^{(K-1)} \right)^2}{2} - \Omega_2^{(K-1)}.
\end{split}
\end{equation*}
If one of $f^{(K-1)}$ and $F^{(K-1)}$ is $\infty$ then the other one is not, and thus the terms $\mp \frac{\Lambda}{2}$ in Lemma~\ref{lem:omega_2} cancel each other out. Since in this case $\mathscr{P}(q^{(K-1)}, f^{(K-1)}, F^{(K-1)})$ and $\mathscr{P}(q^{(K)}, f^{(K)}, F^{(K)})$ are isospectral, we have $\Trace(q^{(K-1)}, f^{(K-1)}, F^{(K-1)}) = \Trace(q^{(K)}, f^{(K)}, F^{(K)})$. Applying now Lemma~\ref{lem:omega_2}, we arrive at the same value for $\Trace(q^{(K-1)}, f^{(K-1)}, F^{(K-1)})$. Repeating this argument $K - 1$ more times concludes the proof.
\end{proof}

\subsection{Inverse problem by spectral data} \label{ss:inverseproblems}

Inverse eigenvalue problems for the Sturm--Liouville equation with boundary conditions dependent on the eigenvalue parameter have been studied in many works. Most of them consider the case of linear dependence on the eigenvalue parameter (see, e.g., \cite{AOK2009}, \cite{BP1980}, \cite{G2005a}, \cite{K2004}, \cite{MC2014}). More general boundary conditions have also been studied (see, e.g., \cite{ABHM2007}, \cite{BBW2002b}, \cite{CF2009}, \cite{C2001}, \cite{FY2010}). Problems with coupled boundary conditions dependent on the eigenvalue parameter are considered in \cite{IN2016}, \cite{SSA2016}.

Throughout this and the next subsection, we consider problems $\mathscr{P}(q, f, F)$ with $q \in \mathscr{L}_2(0, \pi)$, since in this case the necessary and sufficient conditions for the solvability of the inverse problem are especially elegant (see also Remark~\ref{rem:L_1} at the end of the subsection). Theorem~\ref{thm:asymptotics} shows that the spectral data of a problem of the form (\ref{eq:SL})-(\ref{eq:boundary}) necessarily satisfies the conditions
\begin{equation} \label{eq:increasing}
  \lambda_0 < \lambda_1 < \lambda_2 < \ldots, \qquad \gamma_n > 0, \quad n \ge 0
\end{equation}
and
\begin{equation} \label{eq:asymptotics}
\begin{aligned}
  \sqrt{\lambda_n} &= n - \frac{M + N}{2} + \frac{\sigma}{\pi n} + \ell_2 \left( \frac{1}{n} \right), \\
  \gamma_n &= \frac{\pi}{2} \left( n - \frac{M + N}{2} \right)^{2M} \left( 1 + \ell_2 \left( \frac{1}{n} \right) \right)
\end{aligned}
\end{equation}
for some real $\sigma$ and integers $M$, $N \ge -1$. The aim of this subsection is to prove that these necessary conditions are also sufficient for sequences of real numbers $\{ \lambda_n \}_{n \ge 0}$ and $\{ \gamma_n \}_{n \ge 0}$ to be the eigenvalues and the norming constants of a problem of the form (\ref{eq:SL})-(\ref{eq:boundary}).

It is well known that for sequences of real numbers $\{ \lambda_n \}_{n \ge 0}$ and $\{ \gamma_n \}_{n \ge 0}$ satisfying these conditions with $-1 \le M$, $N \le 0$, there exists a unique boundary value problem $\mathscr{P}(q, f, F)$ with constant boundary conditions having these sequences as its spectral data (see, e.g., \cite[Theorem 1.5.2 and Remark 1.5.1]{FY2001}). In this case $M = -1$ (respectively, $N = -1$) if and only if $f = \infty$ (respectively, $F = \infty$), i.e., if and only if the boundary condition at $0$ (respectively, at $\pi$) is Dirichlet. The transformations defined in Section~\ref{sec:transformations} allow us to extend this result to the case of boundary conditions (\ref{eq:boundary}).

\begin{theorem} \label{thm:by_spectral_data}
Let $\{ \lambda_n \}_{n \ge 0}$ and $\{ \gamma_n \}_{n \ge 0}$ be sequences of real numbers satisfying the conditions (\ref{eq:increasing}) and (\ref{eq:asymptotics}). Then there exists a unique boundary value problem $\mathscr{P}(q, f, F)$ having the spectral data $\{ \lambda_n, \gamma_n \}_{n \ge 0}$.
\end{theorem}
\begin{proof}
Denote $K := \max \{ M, N \}$, and consider the numbers $M^{(k)}$, $N^{(k)}$ and the sequences $\{ \lambda_n^{(k)} \}_{n \ge 0}$, $\{ \gamma_n^{(k)} \}_{n \ge 0}$ for $k = 0$, $1$, $\ldots$, $K$ defined by
$$
  M^{(0)} := M, \qquad N^{(0)} := N, \qquad \lambda_n^{(0)} := \lambda_n, \qquad \gamma_n^{(0)} := \gamma_n
$$
and
$$\begin{aligned}
  M^{(k)} &:= M^{(k-1)} - I, \qquad & N^{(k)} &:= N^{(k-1)} + I - 2 J, \\
  \lambda_n^{(k)} &:= \lambda_{n+J}^{(k-1)}, \qquad & \gamma_n^{(k)} &:= \frac{\gamma_{n+J}^{(k-1)}}{(\lambda_{n+J}^{(k-1)} - \lambda_0^{(k-1)} + 2 - 2 J)^I},
\end{aligned}$$
where
$$
  I := \begin{cases} 1, & M^{(k-1)} \ge 0, \\ -1, & M^{(k-1)} = -1, \end{cases} \qquad J := \begin{cases} 1, & M^{(k-1)}, N^{(k-1)} \ge 0, \\ 0, & \text{otherwise} \end{cases}
$$
(we omit the indices of $I$ and $J$ to avoid double indices). It is easy to see that they satisfy the conditions (\ref{eq:increasing}) and (\ref{eq:asymptotics}) with $M$, $N$, $\lambda_n$ and $\gamma_n$ replaced by $M^{(k)}$, $N^{(k)}$, $\lambda_n^{(k)}$ and $\gamma_n^{(k)}$ respectively. Moreover, one of the numbers $M^{(K)}$ and $N^{(K)}$ is $0$, while the other one is either $0$ or $-1$. Hence there exists a boundary value problem $\mathscr{P}(q^{(K)}, f^{(K)}, F^{(K)})$ (with constant boundary conditions) having $\{ \lambda_n^{(K)}, \gamma_n^{(K)} \}_{n \ge 0}$ as its spectral data. Now we successively define $\mathscr{P}(q^{(K-1)}, f^{(K-1)}, F^{(K-1)})$, $\dots$, $\mathscr{P}(q^{(0)}, f^{(0)}, F^{(0)})$ by
$$
  (q^{(k-1)}, f^{(k-1)}, F^{(k-1)}) := \widetilde{\mathbf{T}}(\lambda_0^{(k-1)}, \gamma_0^{(k-1)}, q^{(k)}, f^{(k)}, F^{(k)}).
$$
Theorem~\ref{thm:inverse_transformation} ensures at each step that the spectral data of $\mathscr{P}(q^{(k)}, f^{(k)}, F^{(k)})$ is $\{ \lambda_n^{(k)}, \gamma_n^{(k)} \}_{n \ge 0}$, and hence the existence part of the theorem follows.

To prove the uniqueness part assume that $\mathscr{P}(q, f, F)$ and $\mathscr{P}(\widetilde{q}, \widetilde{f}, \widetilde{F})$ have the same spectral data. Theorem~\ref{thm:asymptotics} implies $\ind f = \ind \widetilde{f}$ and $\ind F = \ind \widetilde{F}$. Denote $K := \max \{ \ind f, \ind F \}$. Together with the problems $\mathscr{P}(q^{(k)}, f^{(k)}, F^{(k)})$ defined by (\ref{eq:P_k}) we consider the problems $\mathscr{P}(\widetilde{q}^{(k)}, \widetilde{f}^{(k)}, \widetilde{F}^{(k)})$ defined by
\begin{equation} \label{eq:P_tilde_k}
\begin{aligned}
  (\widetilde{q}^{(0)}, \widetilde{f}^{(0)}, \widetilde{F}^{(0)}) &:= (\widetilde{q}, \widetilde{f}, \widetilde{F}), \\
  (\widetilde{q}^{(k)}, \widetilde{f}^{(k)}, \widetilde{F}^{(k)}) &:= \widehat{\mathbf{T}} (\widetilde{q}^{(k-1)}, \widetilde{f}^{(k-1)}, \widetilde{F}^{(k-1)}), \qquad k = 1, 2, \ldots, K.
\end{aligned}
\end{equation}
It follows from Theorem~\ref{thm:transformation} that $\mathscr{P}(q^{(k)}, f^{(k)}, F^{(k)})$ and $\mathscr{P}(\widetilde{q}^{(k)}, \widetilde{f}^{(k)}, \widetilde{F}^{(k)})$ have the same spectral data. In particular, $\mathscr{P}(q^{(K)}, f^{(K)}, F^{(K)})$ and $\mathscr{P}(\widetilde{q}^{(K)}, \widetilde{f}^{(K)}, \widetilde{F}^{(K)})$ are two problems with constant boundary conditions and the same spectral data. Therefore $(q^{(K)}, f^{(K)}, F^{(K)}) = (\widetilde{q}^{(K)}, \widetilde{f}^{(K)}, \widetilde{F}^{(K)})$. Finally, successive applications of Theorem~\ref{thm:inverse} concludes the proof.
\end{proof}

\begin{remark} \label{rem:L_1}
A similar characterization can be obtained for potentials $q \in \mathscr{L}_1(0, \pi)$. In this case, however, in addition to the asymptotics of the spectral data, one more condition is needed, namely that a certain function of two variables (the kernel of the Gelfand--Levitan--Marchenko equation) has summable first derivatives (see \cite[Theorem 1.6.1]{LG1964}). For problems $\mathscr{P}(q, f, F)$ with $\ind f \le 0$ this result can be found in
\cite[Section 6]{BBW2002b}.
\end{remark}

\subsection{Symmetric case} \label{ss:symmetric}

Theorem~\ref{thm:by_spectral_data} shows that the spectrum of a boundary value problem of the form (\ref{eq:SL})-(\ref{eq:boundary}) does not uniquely determine this problem. But the situation is different in the case of \emph{symmetric} boundary value problems, i.e., for $\mathscr{P}(q, f, f)$ with $q(x) = q(\pi - x)$. In this subsection we will prove that the spectrum alone determines the symmetric potential $q$ and the boundary coefficient $f$.

We start by studying the properties of symmetric problems. Theorem~\ref{thm:asymptotics} shows that the eigenvalues of $\mathscr{P}(q, f, f)$ satisfy the asymptotics
\begin{equation} \label{eq:symmetric_asymptotics}
  \sqrt{\lambda_n} = n - M + \frac{\sigma}{\pi n} + \ell_2 \left( \frac{1}{n} \right),
\end{equation}
where $M = \ind f$ and
\begin{equation*}
\sigma = \frac{1}{2} \int_0^\pi q(x) \,\du x + 2 \omega_1.
\end{equation*}
Since our problem is symmetric, it follows from~(\ref{eq:phi_psi}) that $\psi(x, \lambda) = \varphi(\pi - x, \lambda)$. Then~(\ref{eq:beta}) implies $\psi(x, \lambda_n) = \beta_n^2 \psi(x, \lambda_n)$, and hence $\beta_n^2 = 1$. Using Theorem~\ref{thm:oscillation} we obtain $\beta_n = (-1)^n$. Thus Lemma~\ref{lem:chi_beta_gamma} implies
\begin{equation} \label{eq:symmetric_chi_gamma}
  \gamma_n = (-1)^n \chi'(\lambda_n).
\end{equation}
It follows from the asymptotics of $\chi(\lambda)$ (see the proof of Lemma~\ref{lem:asymptotics}) that this function is an entire function of order $1/2$, and hence from Hadamard's theorem (see, e.g., \cite[Section 4.2]{L1996}) we obtain
\begin{equation} \label{eq:product}
  \chi(\lambda) = -\pi \prod_{n=0}^M (\lambda_n - \lambda) \prod_{n=M+1}^\infty \frac{\lambda_n - \lambda}{(n - M)^2}.
\end{equation}

Now we are ready to state the main result of this subsection.

\begin{theorem} \label{thm:symmetric}
Let $\{ \lambda_n \}_{n \ge 0}$ be a strictly increasing sequence of real numbers satisfying the asymptotics (\ref{eq:symmetric_asymptotics}) for some real $\sigma$ and integer $M \ge -1$. Then there exists a unique symmetric boundary value problem $\mathscr{P}(q, f, f)$ having the spectrum $\{ \lambda_n \}_{n \ge 0}$.
\end{theorem}
\begin{proof}
Define $\chi$ by~(\ref{eq:product}) and then $\gamma_n$ by~(\ref{eq:symmetric_chi_gamma}). It follows from \cite[Lemma 3.4.2]{M1977} that $\chi$ has a representation of the form
\begin{equation*}
  \chi(\lambda) = - \left( \frac{\sin \pi \sqrt{\lambda}}{\sqrt{\lambda}} - \frac{4 \sigma \cos \pi \sqrt{\lambda}}{4 \lambda - 1} + \frac{g(\lambda)}{\lambda} \right) \prod_{n=0}^M (\lambda_n - \lambda),
\end{equation*}
where $g(\lambda) = \int_0^\pi \widetilde{g}(t) \cos \sqrt{\lambda} t \,\du t$ for some $\widetilde{g} \in \mathscr{L}_2(0, \pi)$.
This representation and (\ref{eq:product}) imply that $\gamma_n$ are strictly positive numbers satisfying the asymptotics
\begin{equation*}
  \gamma_n = \frac{\pi}{2} (n - M)^{2M} \left( 1 + \ell_2 \left( \frac{1}{n} \right) \right).
\end{equation*}
The rest of the proof now follows from Theorem~\ref{thm:by_spectral_data}.
\end{proof}

\subsection{Inverse problem with partial information on the potential} \label{ss:partialinformation}

Another type of inverse problems where the spectrum alone is sufficient are the so-called problems with partial information on the potential. In the case of constant (i.e., independent of the eigenvalue parameter) boundary conditions Hochstadt and Lieberman \cite{HL1978} showed that the knowledge of the norming constants can be replaced by the knowledge of the potential on half the interval and of the boundary constants. Hald \cite{H1980} proved that one of the boundary constants need not be assumed known, i.e., for a problem $\mathscr{P}(q, h, H)$ with constant boundary conditions ($h$, $H \in \mathbb{R} \cup \{\infty\}$) the knowledge of $q$ on $[0, \pi / 2]$ together with $h$ and the spectrum uniquely determines $H$ and $q$ a.e. on all of $[0, \pi]$ (see \cite[Lemma 1]{H1980}, \cite[Theorem A.1]{GS2000}). Later Gesztesy and Simon \cite{GS2000} and Ramm \cite{R2000} showed that if the potential is known on more than half the interval then only a finite density subset of eigenvalues is needed. See also \cite{AFR2009}, \cite{H2005}, \cite{HS2016}, \cite{MP2016}, \cite{WX2012} for further developments in this direction, and \cite{BP1980}, \cite{W2013}, \cite{YH2010} for boundary conditions dependent on the eigenvalue parameter. Here we generalize the Hochstadt--Lieberman theorem to the case of boundary value problems of the form (\ref{eq:SL})-(\ref{eq:boundary}).

\begin{theorem} \label{thm:partial}
Let $\{ \lambda_n \}_{n \ge 0}$ and $\{ \widetilde{\lambda}_n \}_{n \ge 0}$ denote the eigenvalues of the problems $\mathscr{P}(q, f, F)$ and $\mathscr{P}(\widetilde{q}, f, \widetilde{F})$ respectively, where $q$, $\widetilde{q} \in \mathscr{L}_1(0, \pi)$, $f$, $F$, $\widetilde{F} \in \mathscr{R}$ and $\ind f \ge \ind F$. If $q(x) = \widetilde{q}(x)$ a.e. on $[0, \pi / 2]$ and $\lambda_n = \widetilde{\lambda}_n$ for $n \ge 0$, then $(q, f, F) = (\widetilde{q}, f, \widetilde{F})$.
\end{theorem}
\begin{proof}
From the asymptotics of the eigenvalues (see Theorem~\ref{thm:asymptotics}) we obtain that $\ind F = \ind \widetilde{F}$. Denote $K := \ind f$, and consider the problems $\mathscr{P}(q^{(k)}, f^{(k)}, F^{(k)})$ and $\mathscr{P}(\widetilde{q}^{(k)}, \widetilde{f}^{(k)}, \widetilde{F}^{(k)})$ defined by the formulas (\ref{eq:P_k}) and (\ref{eq:P_tilde_k}) respectively. 
Then $f^{(k)} = \widetilde{f}^{(k)}$ and $\ind f^{(k)} \ge \ind F^{(k)} = \ind \widetilde{F}^{(k)}$ for each $k = 1$, $2$, $\ldots$, $K$; in particular, $f^{(K)} \in \mathscr{R}_0$ and $F^{(K)} \in \mathscr{R}_{-1} \cup \mathscr{R}_0$. Denote by $v^{(k)}(x)$ and $\widetilde{v}^{(k)}(x)$ the solutions of the equations $-y''(x) + q^{(k)}(x)y(x) = \Lambda^{(k)} y(x)$ and $-y''(x) + \widetilde{q}^{(k)}(x)y(x) = \Lambda^{(k)} y(x)$ respectively satisfying the initial conditions
$$
  v^{(k)}(0) = \widetilde{v}^{(k)}(0) = f^{(k)}_\downarrow(\Lambda^{(k)}), \qquad \left( v^{(k)} \right)'(0) = \left( \widetilde{v}^{(k)} \right)'(0) = -f^{(k)}_\uparrow(\Lambda^{(k)}),
$$
where
$$
  \Lambda^{(k)} := \begin{cases} \mathring{\boldsymbol{\uplambda}}(q^{(k)}, f^{(k)}, F^{(k)}), & F^{(k)} \ne \infty, \\ \mathring{\boldsymbol{\uplambda}}(q^{(k)}, f^{(k)}, F^{(k)}) - 2, & F^{(k)} = \infty. \end{cases}
$$
Using the definition of the transformation $\widehat{\mathbf{T}}$, for each $k = 0$, $1$, $\ldots$, $K-1$ we successively obtain $v^{(k)}(x) = \widetilde{v}^{(k)}(x)$ on $[0, \pi / 2]$ and $q^{(k+1)}(x) = \widetilde{q}^{(k+1)}(x)$ a.e. on $[0, \pi / 2]$. Then the problems $\mathscr{P}(q^{(K)}, f^{(K)}, F^{(K)})$ and $\mathscr{P}(\widetilde{q}^{(K)}, f^{(K)}, \widetilde{F}^{(K)})$ with constant boundary conditions satisfy the conditions of the Hochstadt--Lieberman theorem, and thus $(q^{(K)}, f^{(K)}, F^{(K)}) = (\widetilde{q}^{(K)}, f^{(K)}, \widetilde{F}^{(K)})$.
We now observe that both $1 / v^{(K-1)}$ and $1 / \widetilde{v}^{(K-1)}$ satisfy the equation $-y''(x) + q^{(K)}(x)y(x) = \Lambda^{(K-1)} y(x)$ and the same initial conditions at $0$. Hence $v^{(K-1)}(x) = \widetilde{v}^{(K-1)}(x)$ on all of $[0, \pi]$. Thus
$$
  q^{(K-1)} = q^{(K)} + 2 \left( \frac{\left( v^{(K-1)} \right)'}{v^{(K-1)}} \right)' = q^{(K)} + 2 \left( \frac{\left( \widetilde{v}^{(K-1)} \right)'}{\widetilde{v}^{(K-1)}} \right)' = \widetilde{q}^{(K-1)}
$$
a.e. on $[0, \pi]$, and
\begin{equation*}
\begin{split}
  F^{(K-1)} &= \boldsymbol{\Theta} \left( \Lambda^{(K-1)}, -\frac{\left( v^{(K-1)} \right)'(\pi)}{v^{(K-1)}(\pi)}, F^{(K)} \right) \\
  &= \boldsymbol{\Theta} \left( \Lambda^{(K-1)}, -\frac{\left( \widetilde{v}^{(K-1)} \right)'(\pi)}{\widetilde{v}^{(K-1)}(\pi)}, F^{(K)} \right) = \widetilde{F}^{(K-1)}.
\end{split}
\end{equation*}
Repeating this argument $K - 1$ more times concludes the proof.
\end{proof}

\begin{remark} \label{rem:partial}
The condition $\ind f \ge \ind F$ also appears in the above-cited works~\cite{BP1980}, \cite{W2013}, \cite{YH2010}. The question of whether this condition can be removed seems to be open.

One can also try to apply the transformations $\widehat{\mathbf{T}}$ and $\widetilde{\mathbf{T}}$ to other results mentioned above. But in order for this to work, a very restrictive additional condition (the equality of the first several eigenvalues) must be imposed, and hence we do not present these results here.
\end{remark}

\subsection{Inverse problem by interior spectral data} \label{ss:interior}

In this final subsection, we consider a new type of uniqueness problem that has appeared relatively recently. Using the main ideas of~\cite{HL1978}, Mochizuki and Trooshin~\cite{MT2001} proved that for a problem $\mathscr{P}(q, h, H)$ with constant boundary conditions ($h$, $H \in \mathbb{R} \cup \{\infty\}$) the knowledge of $h$ and $H$, the spectrum and the values of the function $\varphi'(\pi / 2, \lambda) / \varphi(\pi / 2, \lambda)$ at the eigenvalues uniquely determines $q$ a.e. on $[0, \pi]$. Our transformations allow us to extend this result to the case of boundary conditions of the form~(\ref{eq:boundary}) with equal indices.

\begin{theorem} \label{thm:interior}
Let $\{ \lambda_n \}_{n \ge 0}$ and $\{ \widetilde{\lambda}_n \}_{n \ge 0}$ denote the eigenvalues of the problems $\mathscr{P}(q, f, F)$ and $\mathscr{P}(\widetilde{q}, f, F)$ respectively, and let $y_n$ and $\widetilde{y}_n$ be corresponding eigenfunctions, where $q$, $\widetilde{q} \in \mathscr{L}_1(0, \pi)$, $f$, $F \in \mathscr{R}$ and $\ind f = \ind F$. If $\lambda_n = \widetilde{\lambda}_n$ and $y'_n(\pi / 2) / y_n(\pi / 2) = \widetilde{y}'_n(\pi / 2) / \widetilde{y}_n(\pi / 2)$ for all $n \ge 0$ then $q(x) = \widetilde{q}(x)$ a.e. on $[0, \pi]$.
\end{theorem}
\begin{proof}
Let $K := \ind f = \ind F$, and consider again the problems $\mathscr{P}(q^{(k)}, f^{(k)}, F^{(k)})$ and $\mathscr{P}(\widetilde{q}^{(k)}, \widetilde{f}^{(k)}, \widetilde{F}^{(k)})$ defined by (\ref{eq:P_k}) and (\ref{eq:P_tilde_k}) respectively. Since the indices of $f$ and $F$ are equal, the numbers $J$ defined by (\ref{eq:J}) always equal $1$, and thus both these problems have eigenvalues $\{ \lambda_{n+k} \}_{n \ge 0}$. We also have $f^{(k)} = \widetilde{f}^{(k)}$ and $F^{(k)} = \widetilde{F}^{(k)}$. It follows from the identities in the proof of Lemma~\ref{lem:oscillation} that
$$
  \frac{\left( y^{(k+1)}_n \right)' \left( \dfrac{\pi}{2} \right)}{y^{(k+1)}_n \left( \dfrac{\pi}{2} \right)} = (\lambda_k - \lambda_{n+k+1}) \left( \frac{\left( y^{(k)}_{n+1} \right)' \left( \dfrac{\pi}{2} \right)}{y^{(k)}_{n+1} \left( \dfrac{\pi}{2} \right)} - \frac{\left( y^{(k)}_0 \right)' \left( \dfrac{\pi}{2} \right)}{y^{(k)}_0 \left( \dfrac{\pi}{2} \right)} \right)^{-1} - \frac{\left( y^{(k)}_0 \right)' \left( \dfrac{\pi}{2} \right)}{y^{(k)}_0 \left( \dfrac{\pi}{2} \right)}
$$
for $k = 0$, $\ldots$, $K-1$ and $n = 0$, $1$, $2$, $\ldots$. The same is true for eigenfunctions $\widetilde{y}^{(k)}_n$ of the problems $\mathscr{P}(\widetilde{q}^{(k)}, \widetilde{f}^{(k)}, \widetilde{F}^{(k)})$. Hence
$$
  \frac{\left( y^{(k)}_n \right)' \left( \dfrac{\pi}{2} \right)}{y^{(k)}_n \left( \dfrac{\pi}{2} \right)} = \frac{\left( \widetilde{y}^{(k)}_n \right)' \left( \dfrac{\pi}{2} \right)}{\widetilde{y}^{(k)}_n \left( \dfrac{\pi}{2} \right)}, \qquad n = 0, 1, 2, \ldots
$$
for each $k$. In particular, $\mathscr{P}(q^{(K)}, f^{(K)}, F^{(K)})$ and $\mathscr{P}(\widetilde{q}^{(K)}, f^{(K)}, F^{(K)})$ are two problems with constant boundary conditions, and thus $q^{(K)}(x) = \widetilde{q}^{(K)}(x)$ a.e. on $[0, \pi]$ by the above-mentioned theorem of Mochizuki and Trooshin. Then the functions $1 / y^{(K-1)}_0$ and $1 / \widetilde{y}^{(K-1)}_0$ satisfy the equation $-y''(x) + q^{(K)}(x)y(x) = \lambda_{n+K-1} y(x)$ and the same initial conditions at $0$. Therefore $y^{(K-1)}_0 = \widetilde{y}^{(K-1)}_0$ on all of $[0, \pi]$,
$$
  q^{(K-1)} = q^{(K)} + 2 \left( \frac{\left( y^{(K-1)}_0 \right)'}{y^{(K-1)}_0} \right)' = q^{(K)} + 2 \left( \frac{\left( \widetilde{y}^{(K-1)}_0 \right)'}{\widetilde{y}^{(K-1)}_0} \right)' = \widetilde{q}^{(K-1)},
$$
and repeating this argument $K - 1$ more times concludes the proof.
\end{proof}

\begin{remark} \label{rem:interior}
As in the previous theorem, the question remains open whether the condition $\ind f = \ind F$, which also appears in~\cite{W2013}, can be removed.
\end{remark}

\end{document}